\documentclass{article}    %
\usepackage{arxiv}

\usepackage{bookmark}
\usepackage{amsthm,amssymb,amsmath}
\newtheorem{thm}{Theorem}
\newtheorem{cor}[thm]{Corollary}
\newtheorem{lem}[thm]{Lemma}
\newtheorem{assum}[thm]{Assumption}
\newtheorem{prop}[thm]{Proposition}
\theoremstyle{definition}
\newtheorem{defn}[thm]{Definition}

\newtheorem{prob}[thm]{Problem}

\usepackage{url}
\usepackage{cite}
\usepackage{array}
\usepackage[utf8]{inputenc}
\usepackage{url}
\usepackage{multirow}
\usepackage{mathrsfs}
\usepackage{optidef}
\usepackage{tikz}
\usetikzlibrary{shapes,arrows, intersections, fit}
\usepackage{adjustbox}

\usepackage{xspace}
\usepackage{amsmath,amssymb,amsfonts}

\newcommand{\bs}[1]{\ensuremath{\boldsymbol{#1}}}

\newcommand{\bw}{\ensuremath{\bs w}\xspace}
\newcommand{\bx}{\ensuremath{\bs x}\xspace}

\newcommand{\bW}{\ensuremath{\bs W}\xspace}
\newcommand{\bX}{\ensuremath{\bs X}\xspace}

\newcommand{\bfeta}{\ensuremath{\bs \eta}\xspace}

\newcommand{\pci}[1]{\ensuremath{\bs{p}_{\boldsymbol{c}_{i}}}\xspace}

\usepackage{algorithm}  %
\usepackage{algpseudocode}%

\makeatletter
\define@key{Gin}{Trim}
            {\let\Gin@viewport@code\Gin@trim\expandafter\Gread@parse@vp#1 \\}
\makeatother

\newcommand{\usc}{u.s.c.}
\newcommand{\Prob}{\mathbb{P}}

\newcommand{\sigmaAlg}{\mathscr{B}}
\newcommand{\targettube}{\mathscr{T}}
\newcommand{\rapi}{r_{\overline{x}_0}^\pi(\targettube)}
\newcommand{\rapiopt}{r_{\overline{x}_0}^{\pi^\ast}(\targettube)}
\newcommand{\rarho}{r_{\overline{x}_0}^\rho(\targettube)}

\newcommand{\xmax}{\overline{x}_\mathrm{max}}
\newcommand{\xanchor}{\overline{x}_\mathrm{anchor}}

\newcommand{\rarhoast}{r_{\overline{x}_0}^{\rho^\ast}(\targettube)}
\newcommand{\SRA}{\mathcal{L}^{\pi^\ast}}
\newcommand{\SRAset}{\SRA_0(\alpha, \targettube)}

\newcommand{\SRAsetkOne}{\SRA_k(\alpha_1, \targettube)}
\newcommand{\SRAsetkTwo}{\SRA_k(\alpha_2, \targettube)}
\newcommand{\SRAsetk}{\SRA_k(\alpha, \targettube)}
\newcommand{\SRAsetbeta}{\SRA_0(\beta, \targettube)}
\newcommand{\SRAsetbetai}{\SRA_0(\beta_i, \targettube)}
\newcommand{\SRAsetkbeta}{\SRA_k(\beta, \targettube)}

\newcommand{\FSRA}{\mathcal{K}^{\rho^\ast}}
\newcommand{\FSRAset}{\FSRA_0(\alpha, \targettube)}

\newcommand{\uFSRAset}{\underline{\FSRA_0}(\alpha, \targettube, \mathcal{D})}
\newcommand{\uFSRAsetcdot}{\underline{\FSRA_0}(\cdot, \targettube, \mathcal{D})}
\newcommand{\uFSRAsetOne}{\underline{\FSRA_0}(\alpha_1, \targettube, \mathcal{D}_1)}
\newcommand{\uFSRAsetTwo}{\underline{\FSRA_0}(\alpha_2, \targettube, \mathcal{D}_2)}

\newcommand{\FSRAsetbeta}{\FSRA_0(\beta, \targettube)}

\newcommand{\psibwk}{\psi_{\bw,k}}
\newcommand{\uDir}{\bar{d}}
\newcommand{\Npoly}{\vert \mathcal{D}\vert}

\newcommand{\conv}[1]{\underset{i\in \mathbb{N}_{[1,#1]}}{\mathrm{conv}}}
\newcommand{\convij}[2]{\underset{i\in \mathbb{N}_{[1,#1]},j\in \mathbb{N}_{[1,#2]}}{\mathrm{conv}}}
\newcommand{\ft}{Fourier transform}

\title{Stochastic reachability of a target tube: Theory and computation }

\author{Abraham P. Vinod\\
    Mitsubishi Electric Research Laboratories, \\
    Cambridge, MA 02139 USA\\
       \texttt{aby.vinod@gmail.com} \\
       \And
       Meeko M. K. Oishi \\
    Electrical and Computer Engineering,\\
    University of New Mexico,\\
    Albuquerque, NM 87131 USA\\
       \texttt{oishi@unm.edu}
    \thanks{This material is based upon work supported by
        the National Science Foundation. Vinod and Oishi are
        supported under Grant Number CMMI-1254990 (CAREER,
        Oishi), CNS-1329878, and IIS-1528047. Any opinions,
        findings, and conclusions or recommendations
        expressed in this material are those of the authors
        and do not necessarily reflect the views of the
        National Science Foundation.\newline
        This work was completed while Abraham P.  Vinod was
        a doctoral student at the University of New
        Mexico.\newline
        Figure~\ref{fig:cartoon_stochr} is licensed by the
        authors under the Creative Commons
        Attribution-ShareAlike 4.0 International License. To
        view a copy of this license, visit
    \url{http://creativecommons.org/licenses/by-sa/4.0/}.}
}

\date{}
\renewcommand{\headeright}{}
\renewcommand{\undertitle}{}

\begin{document}

\maketitle

\begin{abstract}
Probabilistic guarantees of safety and performance are important in constrained dynamical systems with stochastic uncertainty. We consider the stochastic reachability problem, which maximizes the probability that the state remains within time-varying state constraints (i.e., a ``target tube''), despite bounded control authority. This problem subsumes the stochastic viability and terminal hitting-time stochastic reach-avoid problems. Of special interest is the stochastic reach set, the set of all initial states from which it is possible to stay in the target tube with a probability above a desired threshold. We provide sufficient conditions under which the stochastic reach set is closed, compact, and convex, and provide an underapproximative interpolation technique for stochastic reach sets. Utilizing convex optimization, we propose a scalable and grid-free algorithm that computes a polytopic underapproximation of the stochastic reach set and synthesizes an open-loop controller. This algorithm is anytime, i.e., it produces a valid output even on early termination. We demonstrate the efficacy and scalability of our approach on several numerical examples, and show that our algorithm outperforms existing software tools for verification of linear systems.
\end{abstract}


\section{Introduction}
\label{sec:introduction}

Guarantees of safety and performance are crucial when hard constraints must be 
maintained in stochastic dynamical systems, and have application to problems in 
robotics, biomedical applications, and spacecraft 
\cite{SummersAutomatica2010,AbateAutomatica2008,VinodHSCC2018,VinodLCSS2017,MaloneHSCC2014,LesserCDC2013,GleasonCDC2017,soudjani2014probabilistic}.
Indeed, in safety-critical or expensive applications, it is often useful to pose the question: 
\emph{What are the initial set of states that can remain within a desired, time-varying set with at least a desired likelihood, despite bounded control authority?}. 
Stochastic reachability provides a well established mathematical framework \cite{AbateAutomatica2008,SummersAutomatica2010}
to address such questions. %
{\em In this paper, we characterize the set-theoretic properties of stochastic reach sets, and use these properties to construct tractable computational approaches to compute this set and its corresponding control.}
We focus in particular on %
time-varying sets, i.e., stochastic reachability over a target tube, which
subsumes stochastic viability and terminal hitting-time
stochastic reach-avoid problems~\cite{AbateAutomatica2008,
SummersAutomatica2010}.

The theoretical framework for stochastic reachability is based on
dynamic programming \cite{AbateAutomatica2008,
SummersAutomatica2010}, in which
stochastic reach sets are represented as superlevel sets of the optimal value function.
The computational intractability associated with dynamic programming has spurred 
a variety of approaches to compute stochastic reach sets as
well as the stochastic reach probability, including 
approximate dynamic programming \cite{KariotoglouSCL2016, ManganiniCYB2015},
abstraction-based
techniques~\cite{soudjani2014probabilistic,cauchi2019stochy,soudjani2015faust,cauchi2019efficiency},
Gaussian mixtures \cite{KariotoglouSCL2016}, particle filters
\cite{ManganiniCYB2015, LesserCDC2013,sartipizadeh2019}, convex
chance constrained optimization \cite{LesserCDC2013,VinodACC2019}, Fourier
transforms~\cite{VinodLCSS2017,VinodHSCC2018}, and semi-definite
programming~\cite{SDP_kariotgolou, kariotoglou2017linear}. 
The problem of robust reachability of a target tube \cite{bertsekas1971minimax, bertsekas_infinite_1972}
is closely related, but restricts the uncertainty
in the dynamics to lie in a bounded set.  Set-theoretic (i.e., Lagrangian) 
approaches have been proposed for this problem using computational
geometry \cite{bertsekas1971minimax,GleasonCDC2017}, however, reliance upon 
vertex-facet enumeration precludes computation on problems with large time horizons, or with target tubes
that have small sets. Lastly, the stochastic model predictive control community has focused 
on a similar problem, synthesizing (sub)optimal controllers in stochastic systems via receding horizon framework
~\cite{chatterjee2011stochastic,
mesbah2016stochastic,farina2016stochastic,OnoCDC2008,VinodACC2019}.  However, these approaches do not directly 
solve for the stochastic reach set, which is the focus of this paper.

In this paper, we exploit set-based properties to enable computational tractability for stochastic, linear time-varying dynamical systems with time-varying constraint sets.  
We employ 
upper semi-continuity and log-concavity 
to propose sufficient conditions for
existence, convexity, and compactness of stochastic reach sets.  While these approaches apply to nonlinear dynamical systems as well, computational tractability requires convexity associated with linear dynamics.  
Related work in well-posedness and existence
of stochastic reach sets has been developed for the stochastic reach-avoid problem, 
using continuity of the stochastic kernel \cite{DingAutomatica2013,
kariotoglou2017linear, VinodLCSS2017, chatterjee2011maximizing,
YangAutomatica2017}.

We use these set-based properties to construct a) underapproximations of stochastic reach sets, that can be computed in a grid-free manner, and b) a grid-free interpolation scheme, that enables memory-efficient scaling of pre-computed stochastic reach sets.  In concert, these computational tools facilitate near run-time applications of stochastic reachability. 
{\em The main contributions of this paper are: 
1) characterization of sufficient conditions under which stochastic reachability of a target tube is well-defined, and the stochastic reach sets are closed, compact, and convex, 
2) an underapproximative interpolation technique for stochastic reach sets, and 
3) design of a scalable, grid-free, anytime algorithm that employs convex optimization to provide an open-loop controller-based underapproximation, and is assured to provide a valid solution, even if terminated early. }

This paper builds upon our preliminary work \cite{VinodHSCC2018}, which focused on linear time-invariant systems and time-invariant constraint sets, and employed a computational approach based on Fourier transforms \cite{VinodLCSS2017}.  Here, we relax the reach-avoid constraints,
consider time-varying dynamics that may be nonlinear, and generalize the computational approach, to
include additional methods of chance constraint enforcement
\cite{LesserCDC2013, OnoCDC2008,VinodACC2019}.

The rest of this paper is organized as follows.  Section~\ref{sec:prelim}
describes the stochastic reachability problem and relevant properties from
probability theory and real analysis.  Section~\ref{sec:theory} presents
sufficient conditions to guarantee existence, closedness, compactness, and
convexity of the stochastic reach sets, and proposes an underapproximative
interpolation technique based on the proposed convexity results.
In Section~\ref{sec:OL}, we construct open-loop controllers based on the underapproximation,
and propose algorithms for underapproximative set computation and interpolation.
We demonstrate our algorithms on several numerical examples in
Section~\ref{sec:num}, and conclude in Section \ref{sec:conc}.

\section{Preliminaries and problem formulation}
\label{sec:prelim}

We denote the set of natural numbers by $\mathbb{N}$, and the set of real numbers by $\mathbb{R}$, the Borel $\sigma$-algebra by $ \sigmaAlg(\cdot)$, random vectors with
bold case, non-random vectors with an overline, and a discrete-time time
interval which inclusively enumerates all natural numbers in between $a$ and $b$ for
$a,b\in \mathbb{N}$ and $a\leq b$ by $ \mathbb{N}_{[a,b]}$.  The indicator
function of a non-empty set $ \mathcal{E}$ is denoted by
$1_{\mathcal{E}}(\overline{y})$, such that $1_{\mathcal{E}}(\overline{y})=1$ if
$\overline{y}\in \mathcal{E}$ and is $0$ otherwise. For any non-negative $c\in
\mathbb{R}$, the set $c \mathcal{E}=\{c \overline{y}: \overline{y}\in
\mathcal{E}\}$. We denote the affine hull and the convex hull of a set $
\mathcal{E}$ by $\mathrm{affine}(\mathcal{E})$ and $\mathrm{conv}(
\mathcal{E})$, respectively. We denote the Minkowski sum operation using
$\oplus$ --- for any two sets $ \mathcal{E}_1, \mathcal{E}_2$, $ \mathcal{E}_1
\oplus \mathcal{E}_2 = \{ \overline{y}_1+ \overline{y}_2: \overline{y}_1\in
\mathcal{E}_1, \overline{y}_2\in \mathcal{E}_2\}$. We denote the cardinality of
a finite set $ \mathcal{D}$ by $\Npoly$.

\subsection{Real analysis and probability theory}

The relative interior of a set $ \mathcal{E}\subseteq
\mathbb{R}^n$ is defined as
\begin{align}
    \mathrm{relint}( \mathcal{E})&=\{ \overline{x}\in \mathbb{R}^n: \exists r>0,
\mathrm{Ball}( \overline{x},r)\cap \mathrm{affine}( \mathcal{E})\subseteq
\mathcal{E}\}\nonumber %
\end{align}
where $\mathrm{Ball}( \overline{x},r)$ denotes a ball in $ \mathbb{R}^n$
centered at $ \overline{x}$ and of radius $r$ with respect to any Euclidean
norm~\cite[Sec. 2.1.3]{BoydConvex2004}.  The relative interior of a set is
always non-empty. 
The relative boundary is $ \partial
\mathcal{E}=\mathrm{closure}( \mathcal{E})\setminus \mathrm{relint}(
\mathcal{E})$.  
From the Heine-Borel
theorem~\cite[Thm 12.5.7]{TaoAnalysisII}, $ \mathcal{E}$ is compact if and only
if it is closed and bounded.

A function $f: \mathbb{R}^n \rightarrow \mathbb{R}$ is upper semi-continuous (u.s.c.) if its superlevel sets $\{\overline{x}\in \mathbb{R}^n: f(\overline{x}) \geq \alpha\}$ for every $\alpha\in \mathbb{R}$ are closed~\cite[Defs. 2.3 and 2.8]{RudinReal1987}. 
A non-negative function $f: \mathbb{R}^n \rightarrow [0,\infty)$ is log-concave if $\log f$ is concave with $\log 0\triangleq-\infty$~\cite[Sec.  3.5.1]{BoydConvex2004}. 
Since log-concave functions are quasiconcave~\cite[Sec.  3.5.1]{BoydConvex2004}, their superlevel sets are convex.
Many standard distributions are log-concave, for example, Gaussian, uniform, and exponential~\cite[Eg. 3.40]{BoydConvex2004}. 
The indicator function of a closed and convex set is \usc{} and log-concave
respectively~\cite[Eg. 3.1 and Sec.  3.1.7]{BoydConvex2004}.

\subsection{System description}

\begin{figure*}
    \centering
    \includegraphics[width=1\linewidth]{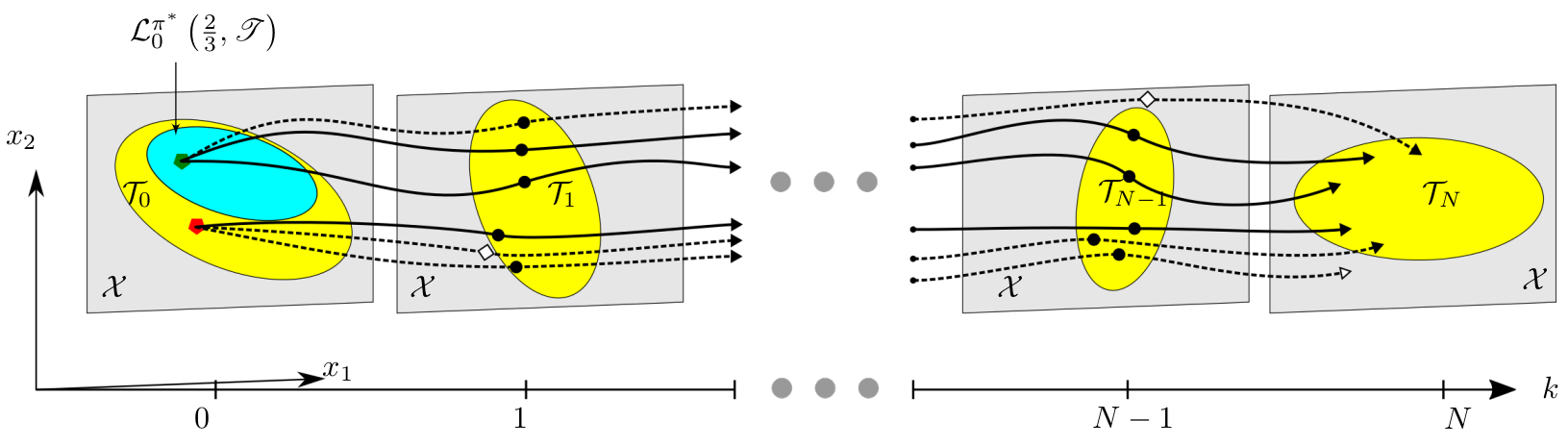}
    \caption{The target tube $\targettube={\{ \mathcal{T}_k\}}_{k=0}^N$, the
        stochastic evolution of \eqref{eq:lin} or \eqref{eq:nonlin} under a maximal reach policy $\pi^\ast$, and the stochastic reach set $\SRAset$ \eqref{eq:stochRAset} for $\alpha=\frac{2}{3}$. 
    Problem \eqref{prob:MP} subsumes the terminal hitting-time reach-avoid
problem~\cite{SummersAutomatica2010} ($\forall k\in
\mathbb{N}_{[0,N-1]},\mathcal{T}_k= \mathcal{S}, \mathcal{T}_N = \mathcal{R}$)
and the viability problem~\cite{AbateHSCC2007,AbateAutomatica2008} ($\forall
k\in \mathbb{N}_{[0,N]},\mathcal{T}_k= \mathcal{S}$) for Borel safe and terminal
sets, $\mathcal{S}$ and $\mathcal{R}$.}
    \label{fig:cartoon_stochr}
\end{figure*}

Consider the discrete-time linear time-varying system,
\begin{align}
    \bx_{k+1}&=A_k\bx_k +B_k \overline{u}_k + \bw_k\label{eq:lin}
\end{align}
with state $\bx_k\in \mathcal{X}= \mathbb{R}^n$, input $ \overline{u}_k \in
\mathcal{U}\subseteq \mathbb{R}^m$, disturbance $\bw_k\in \mathcal{W}\subseteq
\mathbb{R}^n$, a time horizon of interest $N\in \mathbb{N}, N>0$, time-varying
state and input matrices $A_k\in \mathbb{R}^{n\times n}$ and $B_k\in
\mathbb{R}^{n\times m}$ are defined for $k\in \mathbb{N}_{[0,N-1]}$, and an
initial state $ \overline{x}_0 \in \mathcal{X}$.  We assume the input space $
\mathcal{U}$ is compact.  

We assume that the disturbance process ${\{\bw_k\}}_{k=0}^{N-1}$ in \eqref{eq:lin} is
an independent, time-varying random process. We associate with the random vector
$\bw_k$ a probability space $(\mathcal{W}, \sigmaAlg( \mathcal{W}),
\Prob_{\bw,k})$. Here, $\sigmaAlg( \mathcal{W})$ is the collection of Borel
measurable sets of $ \mathcal{W}$, and $\Prob_{\bw,k}$ is the probability measure
associated with $\bw_k$. The concatenated disturbance random vector
$\bW={[\bw_0^\top\ \bw_1^\top\ \cdots\ \bw_{N-1}^\top]}^\top$ is defined in the
probability space $(\mathcal{W}^N, \sigmaAlg( \mathcal{W}^N), \Prob_{\bW})$ with
$\Prob_{\bW}=\prod_{k=0}^{N-1} \Prob_{\bw,k}$. 

The system \eqref{eq:lin} can be equivalently described by a controlled
Markov process with a stochastic kernel that is a time-varying Borel-measurable
function $Q_k: \sigmaAlg( \mathcal{X}) \times \mathcal{X} \times \mathcal{U} \rightarrow [0,1]$.
The stochastic kernel assigns a probability measure on the Borel space
$(\mathcal{X}, \sigmaAlg( \mathcal{X}))$ for the next state $\bx_{k+1}$, parameterized by the
current state $\bx_k$ and current action $\overline{u}_k$. In other words, for
any $ \mathcal{G}\in \sigmaAlg( \mathcal{X})$, $\overline{x} \in \mathcal{X}$,
and $\overline{u}\in \mathcal{U}$, $\int_{\mathcal{G}} Q_k(d\overline{y}\vert
\overline{x}, \overline{u})=\Prob_{\bx}\left\{\bx_{k+1}\in
\mathcal{G}\middle\vert \bx_k=\overline{x}, \overline{u}_k=
\overline{u}\right\}$.
For $\bw_k$ described by a probability density function $\psibwk$,
\begin{align}
    Q_k(d\overline{y}\vert \overline{x},\overline{u})&=
                            \psibwk(\overline{y}-A_k\overline{x}-B_k\overline{u})d\overline{y}.\label{eq:Q_defn_special}
\end{align}

We define a \emph{Markov policy} $\pi=(\mu_0,\mu_1,\ldots,\mu_{N-1})\in
\mathcal{M}$ as a sequence of Borel-measurable state-feedback laws $\mu_k:
\mathcal{X} \rightarrow \mathcal{U}$~\cite[Defn. 2]{AbateAutomatica2008}.  
For a fixed policy $\pi$, the
random vector $\bX= [\bx^\top_1\ \bx^\top_2\ \cdots\ \bx^\top_N]^\top$, defined
in $( \mathcal{X}^N, \sigmaAlg( \mathcal{X}^N),
\Prob_{\bX}^{\overline{x}_0,\pi})$, has a probability measure
$\Prob_{\bX}^{\overline{x}_0,\pi}$ defined using $Q_k$~\cite[Prop.
7.45]{BertsekasSOC1978}. 

We also consider a discrete-time nonlinear time-varying system,
\begin{align}
    \bx_{k+1}&=f_k(\bx_k, \overline{u}_k,\bw_k)\label{eq:nonlin}
\end{align}
with state $\bx_k\in \mathcal{X}\subseteq \mathbb{R}^n$, input $ \overline{u}_k
\in \mathcal{U}\subseteq \mathbb{R}^m$, disturbance $\bw_k\in
\mathcal{W}\subseteq \mathbb{R}^p$, and a time-varying nonlinear function $f_k:
\mathcal{X}\times \mathcal{U}\times \mathcal{W} \rightarrow \mathcal{X}$ defined
for $k\in \mathbb{N}_{[0,N-1]}$. We require $f_k$ to be
Borel-measurable for the state $\bx_k$ to be a well-defined random
process~\cite[Sec. 1.4, Thm.  4]{ChowProbability1997}. The system
\eqref{eq:nonlin} can also be equivalently described by a
controlled Markov process with a stochastic kernel $Q_k$,
\begin{align}
    \int_{\mathcal{G}} Q_k(d\overline{y}\vert \overline{x}, \overline{u})&=\Prob_{\bw,k}\left\{f_k(\overline{x}, \overline{u}, \bw_k)\in
               \mathcal{G}\right\},\label{eq:Q_defn_gen}
\end{align}
for any $ \mathcal{G}\in \sigmaAlg( \mathcal{X})$,
$\overline{x} \in \mathcal{X}$, and $ \overline{u}\in
\mathcal{U}$.
We do not assume $\bw_k$ has a probability density function,
in contrast to \eqref{eq:Q_defn_special}.
See~\cite[Ch. 8]{BertsekasSOC1978} for more details.

\subsection{Stochastic reachability of a target tube}
\label{sub:SRA}

We define the target tube as $ \targettube={\{\mathcal{T}_k\}}_{k=0}^N$ such
that $\mathcal{T}_k \subseteq \mathcal{X}$ are Borel, i.e., $ \mathcal{T}_k\in
\sigmaAlg( \mathcal{X})$. These sets are pre-determined subsets of
$ \mathcal{X}$ that are deemed safe at each time instant within the time horizon.  Define the
\emph{reach probability of a target tube}, $\rapi$, for known $\overline{x}_0$
and $\pi$, as the probability that the evolution of \eqref{eq:lin} or
\eqref{eq:nonlin} under policy $\pi$ lies within the target tube $\targettube$
for the entire time horizon.  Similarly
to~\cite{SummersAutomatica2010,AbateAutomatica2008},
\begin{align}
    \rapi&=
    \Prob_{\bX}^{\overline{x}_0,\pi}\left\{\forall k\in \mathbb{N}_{[0,N]},\ \bx_k\in
    \mathcal{T}_k\right\}. \label{eq:ProbHatR}
\end{align}
As in~\cite[Def. 10]{SummersAutomatica2010}, we define a \emph{maximal
reach policy} as the Markov policy $\pi^\ast$, the optimal solution of
\eqref{prob:MP},
\begin{align}
    \rapiopt=\sup\limits_{\pi\in \mathcal{M}}\rapi.\label{prob:MP}
\end{align} 
Problem \eqref{prob:MP} defines the problem of stochastic reachability of a
target tube, which subsumes the stochastic viability and stochastic reach-avoid
problems~\cite{AbateAutomatica2008, SummersAutomatica2010,
AbateHSCC2007,VinodHSCC2018}. 
The \emph{$\alpha$-level stochastic reach set},
\begin{align}
    \SRAset&= \{\overline{x}_0\in \mathcal{X}: \rapiopt\geq \alpha \},\label{eq:stochRAset}
\end{align}
is the set of initial states which admits a Markov policy that
satisfies the objective of staying within the target tube with a
probability of at least $\alpha$ (Figure~\ref{fig:cartoon_stochr}).

The solution of \eqref{prob:MP} may be characterized via
dynamic programming, a straightforward extension of stochastic reachability~\cite[Thm. 11]{SummersAutomatica2010} and viability~\cite[Thm. 2]{AbateAutomatica2008}.  
Define $V_k^\ast: \mathcal{X} \rightarrow
[0,1],\ k\in \mathbb{N}_{[0,N]}$, by the backward recursion for $\overline{x}\in
\mathcal{X}$,
\begin{subequations}
\begin{align}
    V_N^\ast(\overline{x})&=1_{ \mathcal{T}_N}(\overline{x})\label{eq:VN} \\
    V_k^\ast(\overline{x})&=\sup\limits_{\overline{u}\in \mathcal{U}}1_{
\mathcal{T}_k}(\overline{x})\int_{ \mathcal{X}}
V_{k+1}^\ast(\overline{y})Q_k(d\overline{y}\vert\overline{x},\overline{u}).\label{eq:Vt_recursionQ} 
\end{align}\label{eq:DP_probA}%
\end{subequations}%
Then, the optimal value to \eqref{prob:MP} is $\rapiopt=V_0^\ast(\overline{x}_0)$ 
which assigns to each initial
state $ \overline{x}_0\in \mathcal{X}$ the maximal reach
probability. The maps $V_k^\ast(\cdot)$ are not probability
density functions, since they do not integrate to $1$ over $
\mathcal{X}$.  From \eqref{eq:DP_probA}, we have the
following bounds on $V_k^\ast(\cdot)$,
\begin{align}
    0\leq V_k^\ast( \overline{x})\leq 1_{\mathcal{T}_k}(
    \overline{x}),\quad\forall \overline{x}\in \mathcal{X},\
    \forall k\in \mathbb{N}_{[0,N]}.\label{eq:boundedV}
\end{align}
For $\alpha\in[0,1]$, we define the superlevel sets of $V_k^\ast(\cdot)$ as
\begin{align}
    \SRAsetk&= \{\overline{x}\in \mathcal{X}: V_k^\ast( \overline{x})\geq \alpha
    \}\label{eq:stochRAsetk},
\end{align}
where the $\alpha$-level superlevel set of $V_0^\ast(\cdot)$ coincides with the
$\alpha$-level stochastic reach set \eqref{eq:stochRAset}.

\begin{figure}
    \centering
    \includegraphics[width=0.6\linewidth]{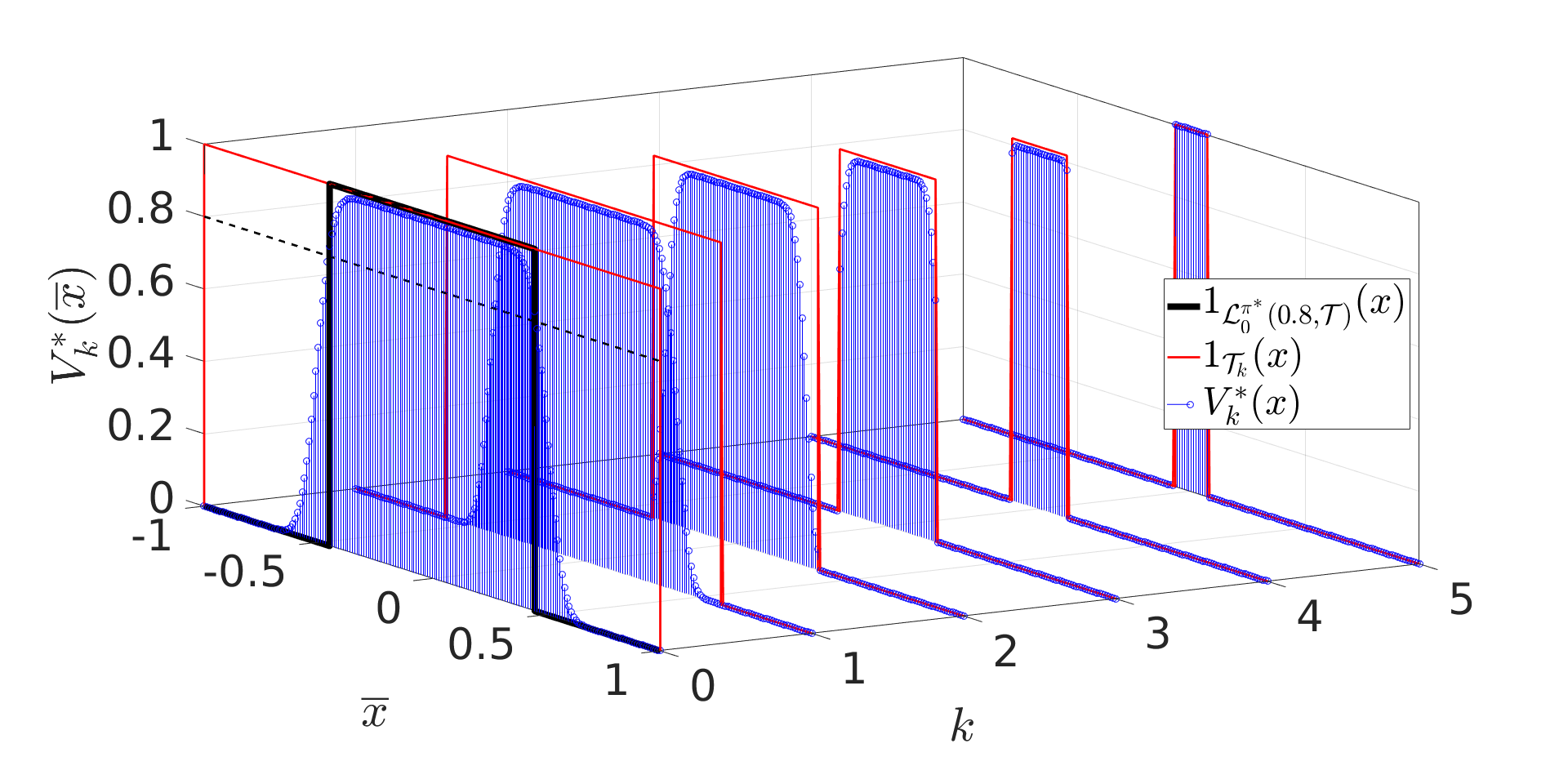}
    \caption{Dynamic programming \eqref{eq:DP_probA} applied
    to example \eqref{eq:ex_sys}. We show the evolution of
$V_k^\ast(x)$ and the construction of the stochastic reach
set $\SRAset$ for $\alpha=0.8$.}
    \label{fig:DP_cartoon}
\end{figure}

\emph{Illustrative example}: Consider the one-dimensional system,
\begin{align}
    \bx_{k+1}&=\bx_k+ u_k +\bw_k\label{eq:ex_sys}
\end{align}
with input $u_k\in[-0.1,0.1]$, and disturbance
$\bw_k\sim \mathcal{N}(0, 0.001)$. Figure~\ref{fig:DP_cartoon} shows the dynamic
programming solution \eqref{eq:DP_probA} for the target tube
$\targettube={\{[-\ell^k, \ell^k]\}}_{k=0}^N$ with
$\ell=0.6$ and time horizon $N=5$. As prescribed by
\eqref{eq:VN}, we set $V_5^\ast(x)=1_{ \mathcal{T}_5}(x)$,
and we compute
$V_k^\ast(\cdot)$ using the backward recursion \eqref{eq:Vt_recursionQ} over a
grid of $\{-1,-0.99,\ldots,0.99,1\}$.  The $0.8$-level stochastic reach set is
given by the superlevel set of $V_0^\ast(\cdot)$ at $0.8$.
As expected, the states near the boundary of the target sets
become increasingly unsafe as we go backward in time.

\subsection{Problem formulation}
\begin{table*}
   \caption{Sufficient conditions for various properties of $V_k^\ast(\cdot)$ and $\SRA(\cdot)$. See~\cite[Thm. 2]{AbateHSCC2007} for Lipschitz continuity of $V_k^\ast(\cdot)$.}
   \label{tab:summary}
    \newcommand{\FunColumnWidth}{1.8cm}
    \newcommand{\SetColumnWidth}{1.5cm}
    \newcommand{\SColumnWidth}{2.5cm}
    \newcommand{\XColumnWidth}{1.1cm}
    \newcommand{\UColumnWidth}{1.5cm}
    \newcommand{\TColumnWidth}{1.5cm}
    \newcommand{\QColumnWidth}{2.8cm}
    \newcommand{\RColumnWidth}{1.5cm}
    \newcommand{\createmycell}[3]{#2}
    \newcommand{\rowadjust}{0.5ex}
    \small
    \adjustbox{width=1\linewidth}{    \begin{tabular}{|m{\FunColumnWidth}|m{\SetColumnWidth}||m{\SColumnWidth}|m{\XColumnWidth}|m{\UColumnWidth}|m{\TColumnWidth}|m{\QColumnWidth} | m{\RColumnWidth}|}
    \hline									
    \multicolumn{2}{|c||}{Property for $k\in \mathbb{N}_{[0,N]}$}&
    \multirow{2}{*}{\begin{minipage}{\SColumnWidth}\centering $f_k$ over $ \mathcal{X}\times \mathcal{U}\times \mathcal{W}$ \end{minipage}} & \multirow{2}{*}{\begin{minipage}{\XColumnWidth}\centering $ \mathcal{X}$ (Borel) \end{minipage}} & \multirow{2}{*}{\begin{minipage}{\UColumnWidth}\centering $ \mathcal{U}$\\ (Compact) \end{minipage}} & \multirow{2}{*}{\begin{minipage}{\TColumnWidth}\centering $ \mathcal{T}_k$ (Borel) \end{minipage}} & \multirow{2}{*}{\begin{minipage}{\QColumnWidth}\centering $ Q_k$ (Borel-measurable) \end{minipage}}  & \multirow{2}{*}{\begin{minipage}{\RColumnWidth}\centering Result \end{minipage}}\\[\rowadjust]\cline{1-2}
    $V_k^\ast(\cdot)$ over $ \mathcal{X}$ & $\SRAsetk$ & & & & & &\\[\rowadjust]\hline\hline
    \createmycell{\FunColumnWidth}{Measurability}{1}  &
    \createmycell{\SetColumnWidth}{Exists $\forall\alpha\in[0,1]$}{1} &
    \createmycell{\SColumnWidth}{Measurable}{2} & -  & - & - &
    \createmycell{\QColumnWidth}{Input-continuous}{1} &
    \createmycell{\RColumnWidth}{Prop.~\ref{prop:Borel}\ref{prop:Borel_exist}}{1} \\[\rowadjust]\cline{1-1}\cline{7-8}
    \createmycell{\FunColumnWidth}{Piecewise continuity}{1} &  & & & & & \createmycell{\QColumnWidth}{Continuous}{1} & \createmycell{\RColumnWidth}{Prop.~\ref{prop:Borel}\ref{prop:Borel_pwcts}}{1} \\[\rowadjust]\hline
    \createmycell{\FunColumnWidth}{Upper semi-continuity}{3}
    & \createmycell{\SetColumnWidth}{Closed
$\forall\alpha\in[0,1]$}{1}  &
\createmycell{\SColumnWidth}{Continuous}{3} &
\createmycell{\XColumnWidth}{Closed}{3} & - &
\createmycell{\TColumnWidth}{Closed}{1} & - & \createmycell{\RColumnWidth}{Prop.~\ref{prop:usc}\ref{prop:usc_closed}}{1} \\[\rowadjust]\cline{2-2}\cline{6-6}\cline{8-8}
     & \createmycell{\SetColumnWidth}{Compact $\forall\alpha\in(0,1]$}{1}  &  &  &  & \createmycell{\TColumnWidth}{Compact}{1} & & \createmycell{\RColumnWidth}{Prop.~\ref{prop:usc}\ref{prop:usc_compact}}{1} \\[\rowadjust]\hline
    \createmycell{\FunColumnWidth}{Log-concavity}{2} &
    \createmycell{\SetColumnWidth}{Convex $\forall\alpha\in[0,1]$}{2}  &
    \createmycell{\SColumnWidth}{Linear \eqref{eq:lin}}{2} &
    \createmycell{\XColumnWidth}{$\mathbb{R}^n$}{2} &
    \createmycell{\UColumnWidth}{Convex}{2} &
    \createmycell{\TColumnWidth}{Convex}{1} &
    \createmycell{\QColumnWidth}{Input-continuous \& Log-concave $\psibwk$}{1}  & \createmycell{\RColumnWidth}{Thm.~\ref{thm:convex}}{2} \\[\rowadjust]\cline{6-7}
                                                     & & & &
                                                     &\createmycell{\TColumnWidth}{Convex
\&\newline closed}{1} & Log-concave  $\psibwk$& \\[\rowadjust]\cline{2-8}
    \createmycell{\FunColumnWidth}{}{} &
    \createmycell{\SetColumnWidth}{Convex \& compact $\forall\alpha\in(0,1]$}{2}  &
    \createmycell{\SColumnWidth}{Linear \eqref{eq:lin}}{2} &
    \createmycell{\XColumnWidth}{$\mathbb{R}^n$}{2} &
    \createmycell{\UColumnWidth}{Convex}{2} &
    \createmycell{\TColumnWidth}{Convex
\&\newline compact}{1} & Log-concave  $\psibwk$ & Prop.~\ref{prop:cvx_cmpt} \\[\rowadjust]\hline
   \end{tabular}}
\end{table*}

\begin{figure}
    \centering
    \tikzstyle{line} = [draw, -latex', thick]
    \begin{tikzpicture}
        \coordinate (model1center) at (0,0);
        \pgfmathsetmacro{\modelOneradius}{3.4}
        \coordinate (model2center) at (2,0);
        \pgfmathsetmacro{\modelTworadius}{3.4}
        \coordinate (model2acenter) at (1.5,-0.4);
        \pgfmathsetmacro{\modelTwoaradius}{1.4}
        \coordinate (model3center) at (1,0);
        \pgfmathsetmacro{\modelThreeradius}{1.5}
        \draw[thick,fill=black!20, name path=Borelcircle] (model1center) circle (\modelOneradius em) node[left, text width=4em, text centered, xshift=-5em] (model1) {\small \emph{Existence}\\(Prop.~\ref{prop:Borel})};
        \draw[thick, fill=blue, fill opacity = 0.5, name path=compactcircle] (model2center) circle (\modelTworadius em) node[fill opacity = 1, right, text width=3.7em, text centered, xshift=4em] (model2) {\small \emph{Existence \& closed}\\(Prop.~\ref{prop:usc}\ref{prop:usc_closed})};
        \draw[thick, fill=red, fill opacity = 0.75, name path=convexcircle] (model3center) circle (\modelThreeradius em) node[fill opacity = 1, above, text width=15em, text centered, yshift=5em] (model3) {\small \emph{Convex} (Thm.~\ref{thm:convex})};
        \draw[thick, fill=green, fill opacity = 0.5] (model2acenter) circle (\modelTwoaradius em) node[fill opacity = 1, below, right, text width=10em, text centered, yshift=-5em,xshift=0.5em] (model2a) {\small \emph{Compactness}\\ (Prop.~\ref{prop:usc}\ref{prop:usc_compact})};
        \node[right of=model2a, text width=11.5em, text centered, xshift=-16em]
            (model4) {\small \emph{Convex \& compact} $ \Rightarrow$
                \emph{Tight\\ polytopic representation}\\
            (Assum.~\ref{assum:cvx_cmpt}, Prop.~\ref{prop:cvx_cmpt})};
        \begin{scope}
          \clip (model2acenter) circle (\modelTwoaradius em);
          \clip (model3center) circle (\modelThreeradius em);
          \fill[thick, black!80, fill opacity = 0.5] (model2acenter) circle (\modelTwoaradius em);
        \end{scope}
        \draw[line] (model1) -- (-0.5,0);
        \draw[line] (model2) -- (2.5,0);
        \draw[line, name path=line 2] (model3) -- (1,0.1);
        \draw[white, line, name intersections={of=line 2 and Borelcircle}] (intersection-1) -- (1,0.1);
        \draw[line] (model2a) -- ([xshift=0.25em,yshift=-0.25em]model2acenter);
        \node [left of=model2acenter,xshift=2.5em,yshift=0.5em] (modelcvxcmpt) {};
        \draw[line, name path=line 1] (model4.north) -- (modelcvxcmpt);
        \draw[white, line, name intersections={of=line 1 and convexcircle}] (intersection-1) -- (modelcvxcmpt);
\end{tikzpicture}
\caption{Various assumptions introduced in Section~\ref{sec:theory}, and the resulting set properties (italicized) of $\SRAsetk$.}
\label{fig:assums_venn}
\end{figure}

We propose grid-free and tractable algorithms to underapproximate $\SRAset$
based on its set-theoretic properties, in lieu of a grid-based implementation of
\eqref{eq:DP_probA}.

\begin{prob}\label{prob:suff_Vt}
    Provide sufficient conditions under which
    $V_k^\ast(\cdot)$ is log-concave and the
    $\alpha$-superlevel set of $V_k^\ast(\cdot)$,
    $\SRAsetk$, is convex for every $k\in
    \mathbb{N}_{[0,N]}$ and $\alpha\in[0,1]$.
\end{prob}
We seek conditions under which the stochastic reachability problem is well-posed, and the stochastic reach set $\SRAsetk$ is closed, bounded, and compact.
Convexity and compactness together enable tight polytopic representations of $\SRAsetk$.
We also seek to exploit these set properties to obtain an underapproximative,
grid-free interpolation scheme for real-time computation.
\begin{prob}~\label{prob:interp}
    For any $\beta\in[\alpha_1,\alpha_2]$ such that
    $0<\alpha_1 < \alpha_2\leq 1$, construct an underapproximative
    interpolation of $\SRAsetkbeta$  using $\SRAsetkOne$ and $\SRAsetkTwo$ for
    any $k\in \mathbb{N}_{[0,N]}$.
\end{prob}

\begin{prob}
    Construct a scalable, grid-free, and anytime algorithm to compute an open-loop
    controller-based polytopic underapproximation to $\SRAsetk,\ \forall k\in
    \mathbb{N}_{[0,N]},\ \forall\alpha\in(0,1]$.\label{prob:compute_poly}
\end{prob}

\section{Stochastic reachability: Properties of \eqref{prob:MP} \& \eqref{eq:stochRAset}}
\label{sec:theory}

The relationship between various assumptions introduced in
this section is shown in Figure~\ref{fig:assums_venn}, and
the set properties are summarized in
Table~\ref{tab:summary}.

\subsection{Existence, continuity, and compactness}
\label{sub:Borel}

We first adapt results from stochastic optimal control to obtain sufficient
conditions under which the stochastic reach set exists, and is closed, bounded,
and compact. For the nonlinear stochastic system \eqref{eq:nonlin} (and thereby
the linear system \eqref{eq:lin}), these
sufficient conditions  guarantee the well-posedness of the stochastic
reachability problem of a target tube \eqref{prob:MP}. The proofs of
Propositions~\ref{prop:contains},~\ref{prop:Borel}, and~\ref{prop:usc} follow
from \eqref{eq:boundedV}, the \emph{measurable selection theorem}~\cite[Thm.
2]{HimmelbergMOR1976}~\cite[Sec.  8.3]{BertsekasSOC1978}, and properties of
\usc{} functions, respectively (see
Appendix~\ref{app:proof_contains}--\ref{app:proof_usc} for the proofs).

\begin{prop}[Bounded]\label{prop:contains}
    Let $\SRAsetk$ exist for some $k\in \mathbb{N}_{[0,N]}$ and $\alpha\in(0,1]$.
    Then, $\SRAsetk\subseteq \mathcal{T}_k$, and bounded $ \mathcal{T}_k$
    implies bounded $ \SRAsetk$.
\end{prop}

\begin{defn}[Continuity of stochastic kernels]~\cite[Defn.
    7.12]{BertsekasSOC1978}\label{defn:cts_stoch}
    A stochastic kernel $Q(\cdot\vert \overline{x},\overline{u})$ is said to be:
    {\renewcommand{\theenumi}{\alph{enumi}}
    \begin{enumerate}
        \item \emph{input-continuous}, if $\int_{ \mathcal{X}} h( \overline{y})
            Q(d\overline{y} \vert \overline{x},\overline{u})$ is continuous
            over $\mathcal{U}$ for each $ \overline{x}\in \mathcal{X}$ for any
            bounded and Borel-measurable function $h: \mathcal{X}\rightarrow \mathbb{R}$.\label{defn:cts_stoch_input}
        \item \emph{continuous}, if $\int_{ \mathcal{X}} h( \overline{y}) Q(d\overline{y} \vert \overline{x},\overline{u})$ is continuous over $ \mathcal{X}\times\mathcal{U}$ for any
            bounded and Borel-measurable function $h: \mathcal{X}\rightarrow \mathbb{R}$.\label{defn:cts_stoch_all}
    \end{enumerate}}%
\end{defn}
From Definition~\ref{defn:cts_stoch}, every continuous
stochastic kernel is also input-continuous.

\begin{prop}[Existence]\label{prop:Borel} Consider
    \eqref{prob:MP} where $f_k(\cdot),\ \forall k\in
    \mathbb{N}_{[0,N-1]}$ is Borel-measurable over $
    \mathcal{X}\times \mathcal{U}\times \mathcal{W}$, $\mathcal{U}$ is
    compact, and $\targettube = {\{
    \mathcal{T}_k\}}_{k=0}^N$ such that $ \mathcal{T}_k\subseteq \mathcal{X},\
    \forall k\in \mathbb{N}_{[0,N]}$ are Borel.
    The following statements are true for every $k\in \mathbb{N}_{[0,N]}$ and
    $\alpha\in[0,1]$:
    {\renewcommand{\theenumi}{\alph{enumi}}
    \begin{enumerate}
        \item if $Q_k$ is input-continuous for every $k\in
            \mathbb{N}_{[0,N]}$, then $\pi^\ast$
            and $\SRAsetk$ exist, and $V_k^\ast(\cdot)$ is
            Borel-measurable. \label{prop:Borel_exist}
        \item if $Q_k$ is continuous for every $k\in
            \mathbb{N}_{[0,N]}$, then $V_k^\ast(\cdot)$
            is additionally piecewise-continuous over $
            \mathcal{X}$ in the
            sense that $V_k^\ast(\cdot)$ is continuous over
            the relative interior and the complement of
            $\mathcal{T}_k$. \label{prop:Borel_pwcts}
    \end{enumerate}}
\end{prop}

\begin{prop}[Closed/compact]\label{prop:usc}
    Consider \eqref{prob:MP} where $f_k(\cdot),\ \forall k\in
    \mathbb{N}_{[0,N-1]}$ is continuous over $ \mathcal{X}\times
    \mathcal{U}\times \mathcal{W}$, $ \mathcal{X}$ is closed, $\mathcal{U}$ is
    compact, and $\targettube = {\{ \mathcal{T}_k\}}_{k=0}^N$ such that $
    \mathcal{T}_k\subseteq \mathcal{X},\ \forall k\in
    \mathbb{N}_{[0,N]}$ are closed.  The following statements are true for every
    $k\in \mathbb{N}_{[0,N]}$ and $\alpha\in [0,1]$:
    {\renewcommand{\theenumi}{\alph{enumi}}
    \begin{enumerate}
        \item $\pi^\ast$ and $\SRAsetk$ exist,
            $V_k^\ast(\cdot)$ is \usc{} over $ \mathcal{X}$,
            and $\SRAsetk$  is
            closed.\label{prop:usc_closed}
        \item additionally, if $ \mathcal{T}_k$ is bounded
            for some $k\in \mathbb{N}_{[0,N]}$, then
            $\SRAsetk$ is compact.\label{prop:usc_compact}
    \end{enumerate}}
\end{prop}

The key difference between assumptions in
Proposition~\ref{prop:Borel} and~\ref{prop:usc} is the
replacement of continuity requirements on $Q_k$, with
stricter requirements on $f_k(\cdot)$, $ \mathcal{X}$, and
$\targettube{}$. These assumptions do not subsume each
other in general, as illustrated in
Figure~\ref{fig:assums_venn}.

\subsection{Convexity and compactness: Polytopic representation of $\SRAsetk$}%
\label{sub:poly_SRAset}

An
underapproximative polytope of a closed and convex set is the convex hull of a
collection of points within the set~\cite[Sec. 2.1.4]{BoydConvex2004}. Using
compactness, this underapproximation can be made tight by
identifying the extreme points in the set \cite[Thm
2.6.16]{webster1994convexity}.
Recall that extreme points of a set are points that cannot
be expressed as a convex combination of any two distinct
points in the set
\cite[Ch. 2]{webster1994convexity}. 
The key benefit of demonstrating convexity of stochastic reach sets is that it
enables a grid-free characterization.

\begin{thm}[Convex]\label{thm:convex}
    Consider \eqref{prob:MP} where the system is linear \eqref{eq:lin}, $
    \mathcal{X} = \mathbb{R}^n$, $\mathcal{U}$ is
    convex and compact, $\targettube = {\{
    \mathcal{T}_k\}}_{k=0}^N$ such that $ \mathcal{T}_k\subseteq \mathcal{X},\ \forall k\in \mathbb{N}_{[0,N]}$
    are convex, and $\psibwk,\ \forall k\in \mathbb{N}_{[0,N]}$ is a log-concave
    probability density function. Further, let $Q_k,\ \forall k\in
    \mathbb{N}_{[0,N-1]}$ be input-continuous, or the set $\mathcal{T}_k,\ \forall k\in
    \mathbb{N}_{[0,N]}$ be closed. Then, $V_k^\ast(\cdot)$ is log-concave over $
    \mathcal{X}$ and $\SRAsetk$ is convex, for
    every $k\in \mathbb{N}_{[0,N]}$ and $\alpha\in[0,1]$.
\end{thm}
\begin{proof}
    The hypothesis of Theorem~\ref{thm:convex} ensures that \eqref{prob:MP} is
    well-posed and an optimal Markov policy exists by
    Propositions~\ref{prop:Borel} or~\ref{prop:usc}. 

    Since $ \mathcal{T}_k$ are convex, the indicator functions $ 1_{\mathcal{T}_k}(\cdot),\ \forall k\in
    \mathbb{N}_{[0,N]}$ are log-concave over $ \mathcal{X}$.
    Hence, $V_N^\ast(\cdot)$ is log-concave over $ \mathcal{X}$ by \eqref{eq:VN}. 
    Next, we will show that $\int_{ \mathcal{X}}
    V_{k+1}^\ast(\overline{y})Q_k(d\overline{y}\vert\overline{x},\overline{u})$
    is log-concave over $ \mathcal{X}\times \mathcal{U}$ for every $k\in
    \mathbb{N}_{[0,N-1]}$, from which the log-concavity of $V_k^\ast(\cdot)$
    will follow.
    By \eqref{eq:Q_defn_special}, for every $k\in \mathbb{N}_{[0,N-1]}$,
    \begin{align}
    \int_{ \mathcal{X}}
    V_{k+1}(\overline{y}) Q_{k}(d\overline{y}\vert \overline{x},
    \overline{u})
    &=\int_{ \mathcal{X}} V_{k+1}(A_k\overline{x} + B_k \overline{u} +
    \overline{w})\psibwk( \overline{w})d\overline{w}. 
    \end{align}
    Our proof is by induction.  Consider the base case
    $k=N-1$.  To show log-concavity of $\int_{ \mathcal{X}}
    V_{N}(\overline{y}) Q_{N-1}(d\overline{y}\vert
    \overline{x}, \overline{u})$, first recall that
    compositions of log-concave functions with affine
    functions preserve log-concavity~\cite[Sec.
    3.2.2]{BoydConvex2004}.  Hence, $V_N(A_{N-1}
    \overline{x}_{N-1} + B \overline{u}_{N-1} +
    \overline{w}_{N-1})$ is log-concave over $
    \mathcal{X}\times \mathcal{U} \times \mathcal{W}$. The
    proof for log-concavity of $\int_{ \mathcal{X}}
    V_{N}^\ast(\overline{y})Q_{N-1}(d\overline{y}\vert\overline{x},\overline{u})$
    is complete with the observation that multiplication and
    partial integration preserve log-concavity~\cite[Sec.
    3.5.2]{BoydConvex2004}.  The proof of log-concavity of
    $V_{N-1}(\cdot)$ follows from the convexity of $
    \mathcal{U}$ and the fact that log-concavity is
    preserved under partial supremum over convex sets and
    multiplication~\cite[Secs. 3.2.5 and
    3.5.2]{BoydConvex2004}.

    For the case $k=t$ with $ t\in \mathbb{N}_{[0,N-2]}$, assume for induction
    that $V_{t+1}^\ast(\cdot)$ is log-concave over $ \mathcal{X}$.  
    The proofs for the log-concavity of $\int_{ \mathcal{X}}
    V_{t+1}^\ast(\overline{y})Q_t(d\overline{y}\vert\overline{x},\overline{u})$
    and $V_t^\ast(\cdot)$ follow by the same arguments as above.

    For every $k\in \mathbb{N}_{[0,N]}$, the log-concavity
    (and thereby, quasiconcavity) 
    of $V_k^\ast(\cdot)$ implies that the set $\SRAsetk$ is
    convex for $\alpha\in[0,1]$.
\end{proof}
Recall that the probability density $\psibwk$ is log-concave
if and only if $\bw_k$ has a log-concave probability measure
and full-dimensional support~\cite[Thm.
2.8]{dharmadhikari1988unimodality}. Further,
input-continuity of $Q_k$ in \eqref{eq:Q_defn_special} is
guaranteed if $\psibwk$ is a continuous function~\cite[Sec.
8.3]{BertsekasSOC1978}.

\begin{assum}\label{assum:cvx_cmpt}
The system model is linear \eqref{eq:lin}, $ \mathcal{X} =
\mathbb{R}^n$, $\mathcal{U}$ is convex and compact, and $\targettube = {\{
\mathcal{T}_k\}}_{k=0}^N$ such that $ \mathcal{T}_k\subseteq \mathcal{X}$ are
convex and compact $\forall k\in \mathbb{N}_{[0,N]}$, and $\psibwk$ is a
log-concave probability density function. 
\end{assum}
\begin{prop}[Convex and compact]\label{prop:cvx_cmpt}
    Under Assumption~\ref{assum:cvx_cmpt}, $\SRAsetk$ are
    convex and compact, for every $k\in \mathbb{N}_{[0,N]}$
    and $\alpha\in(0,1]$.
\end{prop}
\begin{proof}
    By Proposition~\ref{prop:usc} and Theorem~\ref{thm:convex}. 
\end{proof}

Assumption~\ref{assum:cvx_cmpt} combines the assumptions in
Theorem~\ref{thm:convex} and Proposition~\ref{prop:usc}\ref{prop:usc_compact}.
Theorem~\ref{thm:convex} and Proposition~\ref{prop:cvx_cmpt} solve
Problem~\ref{prob:suff_Vt}.  Proposition~\ref{prop:cvx_cmpt} admits a tight
polytopic representation of $\SRAsetk$ for every $k\in \mathbb{N}_{[0,N]}$ and
$\alpha\in(0,1]$. We lose the guarantee of boundedness, and
thereby compactness, for $\alpha=0$
(Proposition~\ref{prop:contains}).

\subsection{Underapproximative interpolation}
\label{sub:interp}

We now consider Problem~\ref{prob:interp}. Given the stochastic
reach sets at two different thresholds $\alpha_1,\alpha_2\in(0,1]$, we wish to
construct an underapproximative interpolation of the stochastic reach set at a
desired threshold $\beta\in[\alpha_1,\alpha_2]$.
\begin{thm}[Interpolation via Minkowski sum]\label{thm:interp}
    Let Assumption~\ref{assum:cvx_cmpt} hold, and let $k\in \mathbb{N}_{[0,N]}$
    and $\alpha_1,\alpha_2\in(0,1],\ \alpha_1<\alpha_2$. Then, for any
    $\beta\in[\alpha_1,\alpha_2]$, 
    \begin{align}
        \gamma \SRAsetkOne \oplus (1-\gamma) \SRAsetkTwo \subseteq
        \SRAsetkbeta\label{eq:underapprox_MinkSum},
    \end{align}
    where
    \begin{align}
        \gamma&=\frac{\log(\alpha_2)-\log(\beta)}{\log(\alpha_2)-\log(\alpha_1)}\in[0,1].\label{eq:gamma_defn}
    \end{align}
\end{thm}
\begin{proof}
    For any $k\in \mathbb{N}_{[0,N]}$ and
    $\gamma$ in \eqref{eq:gamma_defn}, consider some $ \overline{y}\in \gamma \SRAsetkOne \oplus (1-\gamma)
    \SRAsetkTwo$. Specifically, there exists $ \overline{y}_1\in\SRAsetkOne$ and
    $ \overline{y}_2\in\SRAsetkTwo$ such that $ \overline{y}=\gamma
    \overline{y}_1 + (1-\gamma) \overline{y}_2$, $V_k^\ast\left(
    \overline{y}_1\right)\geq \alpha_1>0$, and $V_k^\ast\left(
\overline{y}_2\right)\geq \alpha_2>0$. Recall that
$z^\gamma$ is nondecreasing over positive $z\in \mathbb{R}$.
Consequently,
    \begin{subequations}
    \begin{align}
        {\left(V_k^\ast\left( \overline{y}_1\right)\right)}^\gamma&\geq
    \alpha_1^\gamma>0 \\
        {\left(V_k^\ast\left( \overline{y}_2\right)\right)}^{(1-\gamma)}&\geq
\alpha_2^{(1-\gamma)}>0 \\
{\left(V_k^\ast\left( \overline{y}_1\right)\right)}^\gamma{\left(V_k^\ast\left(
\overline{y}_2\right)\right)}^{(1-\gamma)}&\geq
\alpha_1^\gamma\alpha_2^{(1-\gamma)}.
    \end{align}\label{eq:pow_V_interp}%
    \end{subequations}%
    By the definition of $\gamma$ in \eqref{eq:gamma_defn},
    $\alpha_1^\gamma\alpha_2^{(1-\gamma)}=\beta$.  By log-concavity of
    $V_k^\ast(\cdot)$ (Theorem~\ref{thm:convex}) and \eqref{eq:pow_V_interp}, we
    have $V_k^\ast\left( \overline{y}\right)\geq {\left(V_k^\ast\left(
        \overline{y}_1\right)\right)}^\gamma{\left(V_k^\ast\left(
\overline{y}_2\right)\right)}^{(1-\gamma)} \geq \beta$. Thus, $
\overline{y}\in\SRAsetkbeta$, which completes the proof.
\end{proof}
\begin{cor}\label{cor:interp}
    Let Assumption~\ref{assum:cvx_cmpt} hold.
    Given $\alpha_1,\alpha_2\in(0,1]$, let $\alpha_1<\alpha_2$,
    $\overline{x}_1^{(1)},\ldots,\overline{x}_1^{(K_1)}\in\SRAsetkOne$ and
    $\overline{x}_2^{(1)},\ldots,\overline{x}_2^{(K_2)}\in\SRAsetkTwo$ for some
    $K_1,K_2\in \mathbb{N}, K_1,K_2>0$, and $k\in \mathbb{N}_{[0,N]}$. For any
    $\beta\in[\alpha_1,\alpha_2]$ and  $\gamma$ defined in
    \eqref{eq:gamma_defn}, 
    \begin{align}
        \convij{K_1}{K_2}(\gamma \overline{x}_1^{(i)}+ (1-\gamma) \overline{x}_2^{(j)}) \subseteq
        \SRAsetkbeta \nonumber.%
    \end{align}
\end{cor}
Using Theorem~\ref{thm:interp} and
Corollary~\ref{cor:interp}, we can interpolate stochastic
reach sets with minimal computational effort. These
interpolations also yield convex, underapproximative sets,
since Minkowski sums preserve convexity~\cite[Sec.
2.1]{BoydConvex2004}.

\section{Underapproximative stochastic reachability of a target tube using open-loop controllers}
\label{sec:OL}

Next, we address Problem~\ref{prob:compute_poly}.  We use open-loop controllers
to facilitate the computation of an underapproximation of the maximal reach
probability \eqref{prob:MP} and the stochastic reach set \eqref{eq:stochRAsetk}.
We propose a scalable, grid-free, and anytime algorithm
to compute an underapproximative polytope.  
We use a convex chance constraints
approach~\cite{VinodACC2019,LesserCDC2013} for highly
efficient and scalable polytope computation, which also
lowers the computation cost associated with the \ft{}
approach~\cite{VinodHSCC2018}.  Finally, we use
Corollary~\ref{cor:interp} to propose a tractable,
underapproximative, grid-free, and polytopic interpolation
to stochastic reach sets.

\subsection{Formulation of the optimization problem}

An open-loop policy $\rho: \mathcal{X} \rightarrow \mathcal{U}^N$ provides a
sequence of inputs $\rho(\overline{x}_0)={[\overline{u}_0^\top\
\overline{u}_1^\top\ \cdots\ \overline{u}_{N-1}^\top]}^\top$ for every initial condition $\overline{x}_0$. 
Note that all actions of this policy are contingent only on
the initial state, and not the current state like in a Markov
policy $\pi$.
The random vector describing the extended state $\bX$, under the action of $
\overline{U}=\rho(\overline{x}_0)$, lies in the probability space $(
\mathcal{X}^N, \sigmaAlg( \mathcal{X}^N),
\Prob_{\bX}^{\overline{x}_0,\overline{U}})$.
Under $\rho(\cdot)$, the open-loop reach probability is
\begin{align}
    \rarho&\triangleq 1_{\mathcal{T}_0}( \overline{x}_0)\Prob_{\bX}^{\overline{x}_0,\overline{U}}\left\{\bX\in \bigtimes_{k=1}^N\mathcal{T}_k\right\}. \label{eq:uProbHatR}
\end{align}

With optimal open-loop controller $\rho^\ast$, 
we define the \emph{maximal open-loop reach probability} $W_0^\ast: \mathcal{X} \rightarrow [0,1]$ as
\begin{align}
    W_0^\ast( \overline{x}_0)&\triangleq\rarhoast=\sup_{\rho( \overline{x}_0)=\overline{U}\in \mathcal{U}^N}\rarho. \label{prob:OL}
\end{align}
Similarly to \eqref{eq:DP_probA}, we define $W_k: \mathcal{X}\times
\mathcal{U}^{N-k} \rightarrow [0,1]$ for every $k\in \mathbb{N}_{[0,N-1]}$,
\begingroup
\begin{subequations}
    \begin{align}
        W_k( \overline{x}, \overline{U}_{k:N}) &= 1_{\mathcal{T}_k}(
        \overline{x})\int_{\mathcal{X}} W_{k+1}( \overline{y},
        \overline{U}_{k+1:N})Q_k(d\overline{y} \vert \overline{x}, \overline{u}_k) \label{eq:W_recursion}\\
                                               W_0^\ast( \overline{x}_0)&=\sup_{ \overline{U}\in \mathcal{U}^N} W_0( \overline{x}_0, \overline{U})\label{eq:W_ast}
    \end{align}\label{eq:W_defn}%
\end{subequations}%
\endgroup
where $W_N( \overline{x}, \overline{U}_{N:N}) = 1_{\mathcal{T}_N}(
\overline{x})$. Here, for every $k\in \mathbb{N}_{[0,N-1]}$, $\overline{U}_{k:N}={[\overline{u}_k^\top\
\overline{u}_{k+1}^\top\ \cdots\ \overline{u}_{N-1}^\top]}^\top\in
\mathcal{U}^{N-k}$ with $\overline{U}_{k:N}$, $
\overline{U}_{N-1:N}= \overline{u}_{N-1}$, and $
\overline{U}_{0:N} = \overline{U}$. 
Since $W_0( \overline{x}_0, \overline{U})=\rarho$ for $\rho( \overline{x}_0) = \overline{U}$,
the optimization problems \eqref{prob:OL} and
\eqref{eq:W_ast} are equivalent with identical objectives and constraints.
 
Similarly to \eqref{eq:stochRAset}, we define the $\alpha$-superlevel set of $\rarhoast$, %
\begin{align}
    \FSRAset&= \{\overline{x}_0\in \mathcal{X}: \rarhoast\geq \alpha \}.\label{eq:FstochRAset}
\end{align}
\begin{prop}\label{prop:W_prop}
    Let Assumption~\ref{assum:cvx_cmpt} hold. Then, the following statements are
    true:
    {\renewcommand{\theenumi}{\alph{enumi}}
    \begin{enumerate}
        \item $W_0^\ast(\cdot)$ is well-defined and \usc
            over $ \mathcal{X}$.\label{prop:W_prop_usc}
        \item $W_0^\ast( \overline{x})\leq V_0^\ast( \overline{x}),\
            \forall\overline{x}\in \mathcal{X}$, and $\FSRAset\subseteq
            \SRAset,\ \forall\alpha\in[0,1]$.  \label{prop:W_prop_OL_use}
        \item \eqref{prob:OL} is a log-concave optimization problem, $W_0(\cdot, \cdot)$ is log-concave over $ \mathcal{X}\times \mathcal{U}^N$, and $W_0^\ast(\cdot)$ is log-concave over $ \mathcal{X}$.\label{prop:W_prop_convex}
        \item $\FSRAset,\ \forall \alpha\in(0,1]$ is convex and compact.\label{prop:W_prop_cvx_cmpt}
    \end{enumerate}}
\end{prop}
The proof of Proposition~\ref{prop:W_prop} is given in
Appendix~\ref{app:proof_W_prop}.

\subsection{Construction of a polytopic underapproximation of $\FSRAset$ under Assumption~\ref{assum:cvx_cmpt}}
\label{sub:poly_algo}

Given a finite set $\mathcal{D}\subset \mathcal{X}$ consisting of direction
vectors $\uDir_i$, we compute a polytopic underapproximation of $\FSRAset$ in
three steps (Figure~\ref{fig:CartoonSetComputation}): 
\begin{enumerate}
    \item find $ \xanchor\in \mathcal{T}_0$, referred to as the \emph{anchor
        point}, such that $W_0^\ast( \xanchor)\geq \alpha$; if no such point
        exists, then $\FSRAset=\emptyset$; else, continue to step 2,
    \item obtain relative boundary points of the set $\FSRAset$ via line
        searches from $\xanchor$ along the directions $\uDir_i\in \mathcal{D}$, and 
    \item compute the convex hull of the relative boundary points, i.e., the
    polytope $\uFSRAset$.  
\end{enumerate}
By~\cite[Sec.  2.1.4]{BoydConvex2004} and
Proposition~\ref{prop:W_prop}\ref{prop:W_prop_OL_use}, 
\begin{align}
    \uFSRAset\subseteq\FSRAset\subseteq \SRAset\label{eq:underapprox_chain}.
\end{align}
Thus, we can utilize the compactness and convexity of
$\FSRAset$
(Proposition~\ref{prop:W_prop}\ref{prop:W_prop_cvx_cmpt}) to
obtain a polytopic underapproximation $\uFSRAset$. We have
$\uFSRAset=\FSRAset$ when all of the extreme points of
$\FSRAset$ are discovered in step 2. This is possible for a
polytopic $\FSRAset$~\cite[Thm.
2.6.16]{webster1994convexity}.

\begin{figure}
    \centering
    \includegraphics[width=0.5\linewidth]{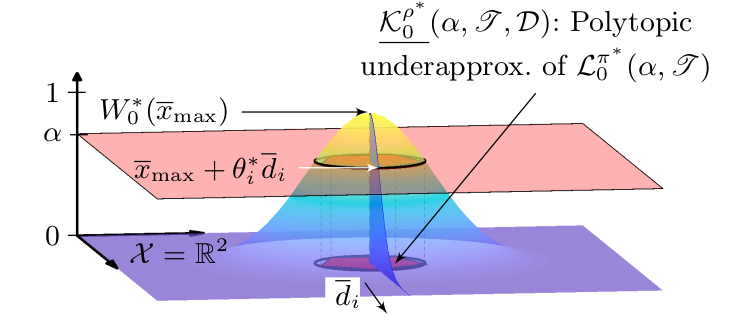}
    \caption{Components of Algorithm~\ref{algo:poly} using $\xanchor=\xmax$. We
    solve \eqref{prob:xmax_compute} to compute $W_0^\ast(
\overline{x}_{\mathrm{max}})$, and then solve \eqref{prob:gamma_ext_W} for each
$ \overline{d}_i\in \mathcal{D}$ to compute the vertices of $\uFSRAset$. We
construct $\uFSRAset$ via the convex hull of the vertices.}
    \label{fig:CartoonSetComputation}
\end{figure}

First, we propose two possible anchor points $ \xanchor$:
1) $\xmax$, the maximizer of $W_0^\ast( \overline{x})$, and 2) $\overline{x}_c$,
a Chebyshev center-like point.  For the former,
the point $\xmax$ can be computed by solving the optimization problem,
\begin{align}
    \setlength\arraycolsep{2pt}
     \begin{array}{rl}
         \underset{\overline{x}, \overline{U}, \phi}{\mathrm{maximize}}& \phi\\
         \mathrm{subject\ to}& \left\{
             \begin{array}{l}
                 \overline{x}\in \mathcal{T}_0,\ \overline{U}\in \mathcal{U}^N,\ \phi\in [0,1],\\
                 W_0( \overline{x},\overline{U})\geq \phi\geq \alpha\\
             \end{array}\right.
    \end{array}
 \label{prob:xmax_compute}
\end{align}
Problem \eqref{prob:xmax_compute} is the epigraph form of the optimization
problem $\sup_{
\overline{x}_0\in \FSRAset} W_0^\ast( \overline{x}_0)$~\cite[Eq.
4.11]{BoydConvex2004}.

For the latter, 
$\overline{x}_c$ can be computed by solving the following
optimization problem, motivated by the Chebyshev centering problem~\cite[Sec.
8.5.1]{BoydConvex2004},
\begin{align}
    \setlength\arraycolsep{2pt}
     \begin{array}{rl}
         \underset{\overline{x}, \overline{U}, R}{\mathrm{maximize}}& R\\
         \mathrm{subject\ to}& \left\{
             \begin{array}{l}
                 \{\overline{x}\} \oplus \mathrm{Ball}(\overline{0},R)\subseteq \mathcal{T}_0\\
                 W_0( \overline{x},\overline{U})\geq \alpha,\ \overline{U}\in
                 \mathcal{U}^N,\ R\geq 0\\
             \end{array}\right.
    \end{array}
 \label{prob:xmax_compute_cheby}
\end{align}
Problem \eqref{prob:xmax_compute_cheby} seeks an anchor point ``deep''
inside $ \mathcal{T}_0$, while ensuring $W_0(
\overline{x},\overline{U})\geq \alpha$ for some $ \overline{U}\in
\mathcal{U}^N$. The centering constraint, $\{\overline{x}\} \oplus
\mathrm{Ball}(\overline{0},R)\subseteq \mathcal{T}_0$, is equivalently expressed
as $\overline{x} + \overline{s}\in \mathcal{T}_0,\ \forall
\overline{s}\in\mathrm{Ball}(\overline{0},R)$. For a polytopic $ \mathcal{T}_0$,
this is a second order-cone (convex) constraint, and can be enforced
efficiently~\cite[Sec. 8.5.1]{BoydConvex2004}.

\begin{prop}\label{prop:cvx_xanchor}
    Under Assumption~\ref{assum:cvx_cmpt}, \eqref{prob:xmax_compute}
    and \eqref{prob:xmax_compute_cheby} are convex optimization problems.
\end{prop}
\begin{proof}
    We know $W_0(\overline{x}, \overline{U})$ is a quasiconcave function over $
    \mathcal{X}\times \mathcal{U}$ by
    Proposition~\ref{prop:W_prop}\ref{prop:W_prop_convex} and the fact that
    log-concave functions are quasiconcave~\cite[Sec. 3.5]{BoydConvex2004}.
    Therefore, $W_0( \overline{x}, \overline{U})\geq \phi$ is a convex
    constraint involving $\overline{x}, \overline{U}$, and $\phi$. Hence
    \eqref{prob:xmax_compute} and \eqref{prob:xmax_compute_cheby} are
    convex, since all other constraints and their respective
    objectives are convex. 
\end{proof}

For non-empty sets $\mathcal{T}_k$ and $\mathcal{U}$,
the optimization problems \eqref{prob:xmax_compute} or
\eqref{prob:xmax_compute_cheby} are infeasible if and only
if $\FSRAset$ is empty. However, $\FSRAset$ being empty does
not imply $\SRAset$ is empty, due to the underapproximative
nature of our approach.

Next, we compute the relative boundary points of $\FSRAset$ by solving
\eqref{prob:gamma_ext_W} for each $i\in \mathbb{N}_{[1,\Npoly]}$,
\begin{align}
    \setlength\arraycolsep{2pt}
     \begin{array}{rl}
         \underset{\theta_i, \overline{U}_i}{\mathrm{maximize}}& \theta_i\\
     \mathrm{subject\ to}& 
        \left\{\begin{array}{l}
                   \overline{U}_i\in \mathcal{U}^N,\quad \theta_i\in \mathbb{R},\quad \theta_i\geq 0\\
                    W_0(\xanchor + \theta_i\uDir_i, \overline{U}_i)\geq\alpha\\
               \end{array}\right.
    \end{array}.\label{prob:gamma_ext_W}
\end{align}%
Problem \eqref{prob:gamma_ext_W} is a line search problem~\cite[Sec.
9.3]{BoydConvex2004}. By Proposition~\ref{prop:cvx_xanchor},
\eqref{prob:gamma_ext_W} is a convex optimization problem under
Assumption~\ref{assum:cvx_cmpt}.

Denote the optimal solution of \eqref{prob:gamma_ext_W} as $\theta^\ast_i$
and $ \overline{U}^\ast_i$.  By construction, $W_0(\xanchor +
\theta_i^\ast\uDir_i, \overline{U}_i^\ast)=\alpha$, and for any $\epsilon>0$,
$W_0(\xanchor + (\theta_i^\ast+\epsilon)\uDir_i, \overline{U}_i^\ast)<\alpha$.
Hence, $\xanchor + \theta_i^\ast \uDir_i$ is a relative boundary point of
$\FSRAset$, and an
open-loop controller that ensures $W_0(\xanchor + \theta_i^\ast\uDir_i,
\overline{U}_i^\ast)\geq\alpha$ is $\rho^\ast(\xanchor + \theta_i^\ast \uDir_i)=
\overline{U}_i^\ast$. Note that $W_0(\xanchor + \theta_i^\ast\uDir_i,
\overline{U}_i^\ast)$ is not necessarily equal to the maximal open-loop reach
probability $W_0^\ast(\xanchor +
\theta_i^\ast\uDir_i)$, since we have only required that
$W_0(\xanchor + \theta_i^\ast\uDir_i,
\overline{U}_i^\ast)\geq \alpha$ in \eqref{prob:gamma_ext_W}.

We construct the polytope $\uFSRAset$ via the convex hull of the
computed relative boundary points $\xanchor + \theta_i^\ast \uDir_i,\ \forall i\in
\mathbb{N}_{[1,\Npoly]}$.
We have $\uFSRAset\subseteq \FSRAset$, since $\FSRAset$ is convex and compact,
and the vertices of $\uFSRAset$ lie in $\partial\FSRAset$~\cite[Sec. 2.1.4]{BoydConvex2004}.

We summarize our approach in Algorithm~\ref{algo:poly}, which solves
Problem~\ref{prob:compute_poly}.

\begin{algorithm}
    \caption{Compute polytope $\uFSRAset$, an underapproximation of the stochastic reach set $\SRAset$}\label{algo:poly}
    \begin{algorithmic}[1]    
        \Require{System \eqref{eq:lin}, target tube
        $\targettube$, probability threshold $\alpha$, set of direction vectors $\mathcal{D}$}
        \Ensure{$\uFSRAset\subseteq \FSRAset\subseteq\SRAset$}     
    \State Solve \eqref{prob:xmax_compute} or \eqref{prob:xmax_compute_cheby} to
    compute $\xanchor$\label{step:xmax}
    \If{\eqref{prob:xmax_compute} or \eqref{prob:xmax_compute_cheby} is infeasible}
        \State $\uFSRAset\gets \emptyset$
    \Else%
        \For{$\uDir_i\in \mathcal{D}$}
            \State Solve \eqref{prob:gamma_ext_W} to compute a relative boundary
            point $\xanchor + \theta_i^\ast \uDir_i$ and an open-loop
            controller $\rho^\ast(\xanchor + \theta_i^\ast \uDir_i)$
	    \EndFor
        \State $\uFSRAset\gets\conv{\Npoly}(\xanchor + \theta_i^\ast \uDir_i)$
    \EndIf
  \end{algorithmic}
\end{algorithm}

\subsection{Implementation of Algorithm~\ref{algo:poly}}
\label{sub:poly_imple}

Algorithm~\ref{algo:poly} is an anytime, parallelizable algorithm. The
anytime property follows from the observation that the convex hull of the
solutions of \eqref{prob:gamma_ext_W} for an arbitrary subset of $ \mathcal{D}$
also yields a valid underapproximation, permitting premature termination.  The
parallelizability of Algorithm~\ref{algo:poly} is due to the fact that the
computations along each of the directions $\uDir_i$ are independent.

Computing $\xanchor$ is a significant part of
Algorithm~\ref{algo:poly}. 
Using $\overline{x}_c$ may be advantageous because 
directions typically yield non-trivial relative
boundary points, however, %
it is possible that  $
\overline{x}_c$ is a relative boundary point of $\uFSRAset$
(e.g., Figure~\ref{fig:CWH}, bottom). Additionally, solving
\eqref{prob:xmax_compute_cheby} requires more computational effort due to
the second-order cone constraint.  However, the ``best'' choice may be 
problem-specific.  Indeed, using multiple anchor
points could be advantageous, as the convex hull of the
union of the resulting underapproximations could yield a
significantly larger underapproximative set.
The implementation of Algorithm~\ref{algo:poly} in
\texttt{SReachTools} provides all three of these options.
\texttt{SReachTools} is an open-source \texttt{MATLAB} toolbox for
stochastic reachability \cite{sreachtools}.
\texttt{SReachTools} has participated in several
repeatability-evaluations including the recent ARCH
initiative for stochastic modeling and
verification~\cite{Abate2020ARCH}.

Denoting the computation times to solve for $\xanchor$
(\eqref{prob:xmax_compute} or \eqref{prob:xmax_compute_cheby}) and \eqref{prob:gamma_ext_W} as $t_{\mathrm{anchor}}$ and $t_{\mathrm{line}}$, respectively, the computation time for Algorithm~\ref{algo:poly} is $ \mathcal{O}(t_{\mathrm{anchor}}+t_{\mathrm{line}}\Npoly)$.
Since \eqref{prob:xmax_compute}, \eqref{prob:xmax_compute_cheby}, and \eqref{prob:gamma_ext_W} are convex problems, globally optimal solutions are assured with (potentially) low $t_{\mathrm{anchor}}$ and $t_{\mathrm{line}}$.
However, the joint chance constraint $W_0(\overline{x}, \overline{U})\geq \phi$ is not solver-friendly, since we do not have a closed-form expression for $W_0(\overline{x}, \overline{U})$, or an exact reformulation into a conic constraint.
In Section~\ref{sub:GaussLTV} (see Table~\ref{tab:implement}), we discuss computationally efficient methods to enforce this constraint under additional assumptions.

The memory requirements of Algorithm~\ref{algo:poly} grow linearly with $\Npoly$
and are independent of the system dimension.
The choice of $ \mathcal{D}$ influences the quality (in terms of volume) of
underapproximation provided by Algorithm~\ref{algo:poly}.  
In contrast, dynamic programming requires an exponential number of grid points in memory, leading to the curse of dimensionality~\cite{AbateHSCC2007}.
Algorithm~\ref{algo:poly} is grid-free and recursion-free, and it scales favorably with the system dimension, as compared to dynamic programming.

\emph{Open-loop controller synthesis}: %
As a side product of Algorithm~\ref{algo:poly}, solving \eqref{prob:gamma_ext_W}
provides valid open-loop controllers for the vertices of
$\uFSRAset$ that satisfy \eqref{eq:FstochRAset}.
Denoting the vertices of $\uFSRAset$ as
$\{\overline{v}_i\}_{i=1}^{|\mathcal{D}|}$, any initial state of
interest $\overline{x}_0\in\uFSRAset$ can be expressed as the convex combination of the vertices, $
\overline{x}_0=\sum_{i=1}^{| \mathcal{D}|}\gamma_i\overline{v}_i$ for some $\gamma_i\in[0,1]$ and
$\sum_{i=1}^{| \mathcal{D}|}\gamma_i=1$~\cite[Ch. 2]{webster1994convexity}.
Consequently, $\sum_{i=1}^{| \mathcal{D}|}\gamma_i\overline{U}_i^\ast$ is a good
initial guess to solve \eqref{prob:OL} at $ \overline{x}_0$.

\subsection{Tractable underapproximative interpolation}
\label{sub:poly_interp}

In scenarios where stochastic reach sets $\uFSRAsetcdot$
must be computed at multiple probability thresholds, we
propose a computationally efficient algorithm that combines
Algorithm~\ref{algo:poly} and Corollary~\ref{cor:interp}.
Specifically, we utilize Algorithm~\ref{algo:poly} to
compute the underapproximations at two specific probability
thresholds $\alpha_1,\alpha_2\in(0,1]$ with
$\alpha_1<\alpha_2$, and then utilize
Corollary~\ref{cor:interp} to obtain underapproximative
interpolations for all  probability thresholds
$\beta\in[\alpha_1,\alpha_2]$. We summarize this approach in
Algorithm~\ref{algo:interp}. Algorithm~\ref{algo:interp}
enables real-time stochastic reachability by computing a few
stochastic reach sets offline and then interpolating them
online.

\begin{algorithm}
    \caption{Underapproximative interpolation}\label{algo:interp}
    \begin{algorithmic}[1]    
        \Require{Probability thresholds $\alpha_1,\alpha_2\in(0,1]$ with
            $\alpha_1< \alpha_2$, sets of direction vectors
            $\mathcal{D}_1$ and $\mathcal{D}_2$, probability thresholds of
            interest
    $\beta_i\in[\alpha_1,\alpha_2]$ for $i\in\mathbb{N}_{[1,N_\beta]}$}
        \Ensure{Polytope $\mathcal{S}_i\subseteq\SRAsetbetai$ for $i\in\mathbb{N}_{[1,N_\beta]}$}     
        \Statex \textbf{Offline (independent of $\beta_{(\cdot)}$)}:
        \State Compute sets $\uFSRAsetOne$ and $\uFSRAsetTwo$ using
        Algorithm~\ref{algo:poly}
        \Statex \textbf{Online (depends on $\beta_{(\cdot)}$)}:
        \setcounter{ALG@line}{0}
        \For{each $\beta_i$}
        \State
        $\gamma_i=\frac{\log(\alpha_2)-\log(\beta_i)}{\log(\alpha_2)-\log(\alpha_1)}\in[0,1]$
        \State $\mathcal{S}_i\gets\gamma_i\uFSRAsetOne\oplus(1-\gamma_i)\uFSRAsetTwo$ 
        \EndFor
  \end{algorithmic}
\end{algorithm}

\subsection{Gaussian linear time-varying systems with polytopic input space and polytopic target tube}
\label{sub:GaussLTV}

\begin{assum}\label{assum:Gauss}
    Presume Assumption~\ref{assum:cvx_cmpt}, polytopic $ \mathcal{U}$ and
    $ \mathcal{T}_k$ for every $k\in \mathbb{N}_{[0,N]}$, and Gaussian $\bw_k\sim \mathcal{N}( \overline{\mu}_{\bw,k}, C_{\bw,k})$, $ \overline{\mu}_{\bw,k}\in \mathbb{R}^n, C_{\bw,k}\in \mathbb{R}^{n\times n}$.
\end{assum}
The concatenated disturbance random vector is $\bW\sim \mathcal{N}( \overline{\mu}_{\bW}, C_{\bW})$, where $ \overline{\mu}_{\bW}={[ \overline{\mu}_{\bw,0}^\top\ \cdots\ \overline{\mu}_{\bw,N-1}^\top]}^\top\in \mathbb{R}^{nN}$ and $C_{\bW}= \mathrm{blkdiag}(C_{\bw,0},\ldots,C_{\bw,N-1})\in \mathbb{R}^{nN\times nN}$, with $ \mathrm{blkdiag}(\cdot)$ indicating block diagonal matrix construction.
Due to the linearity of the system \eqref{eq:lin}, $\bX$ is also Gaussian~\cite[Sec. 9.2]{GubnerProbability2006}.
Given an initial state $ \overline{x}_0\in \mathcal{X}$ and an open-loop vector $ \overline{U}\in \mathcal{U}^N$,
\newcommand{\mubX}{\overline{\mu}_{\bX}}
\newcommand{\CbX}{C_{\bX}}
\begin{subequations}
    \begin{align}
        \bX&\sim \mathcal{N}(\mubX,\CbX) \label{eq:X_U_pdf},\\
        \mubX&=\mathscr{A} \overline{x}_0 + H\overline{U} +G\overline{\mu}_{\bW}\label{eq:X_U_mu},\\
        \CbX&=G C_{\bW} G^\top,\label{eq:X_U_cov}
    \end{align}\label{eq:X_U_dist}%
\end{subequations}%
where $\bX= \mathscr{A} \overline{x}_0 + H \overline{U} + G \bW$.
The matrices $\mathscr{A}, H, G$ account for how the dynamics \eqref{eq:lin}
influence the mean and the covariance of $\bX$ (see~\cite{SkafTAC2010} for
details).

Algorithm~\ref{algo:poly} requires an efficient enforcement of joint chance
constraints, $W_0( \overline{x}, \overline{U})\geq \phi$ in
\eqref{prob:xmax_compute}, \eqref{prob:xmax_compute_cheby}, and
\eqref{prob:gamma_ext_W}.
Under Assumption~\ref{assum:Gauss}, $W_0( \overline{x},
\overline{U})$ is the integration of a Gaussian probability
density function over a polytope. %

\begin{table}
   \centering
   \caption{Enforcing $W_0( \overline{x}_0, \overline{U})\geq \phi$ under Assumption~\ref{assum:Gauss}.}
   \label{tab:implement}
    \begin{tabular}{|m{2.1cm}||m{3cm}|m{2.5cm}|}
    \hline 
    Approach  & Approximation & Solver \\ \hline\hline
    Convex chance\newline constraints\newline\cite{LesserCDC2013,
OnoCDC2008,VinodACC2019} & Convex restriction via Boole's inequality
    & Linear\newline  program
    \\ \hline
    Fourier\newline transform %
    \newline 
    \cite{VinodLCSS2017, VinodHSCC2018} & Approximates $W_0(\cdot)$ via
    Genz's algorithm (quasi-Monte Carlo simulation~\cite{GenzJCGS1992})%
                                        & Gradient-free\newline
                                        optimization
                                        solver~\cite{kolda2003optimization}\newline
                                        (\texttt{patternsearch})\\ \hline
   \end{tabular}%
\end{table}
\begin{figure}
    \centering
    \adjustbox{width=0.6\linewidth}{
    \tikzstyle{block} = [rectangle, draw, text width=5em, text centered, rounded corners, minimum height=2em]
    \tikzstyle{line} = [draw, -latex', very thick]
    \begin{tikzpicture}[node distance = 11em, auto]
        \node [block, text width=9em] (orig) {Stochastic reachability problem \eqref{prob:MP} $\SRAset$};
        \node [block, text width=10em, xshift= 8.5em, right of=orig] (openloop) {Open-loop underapproximation \eqref{prob:OL}  $\FSRAset$};
        \node [block, text width=10em, yshift= 2em, below of=openloop]
            (polytope) {Alg.~\ref{algo:poly}: Polytopic
            \text{underapproximation}\linebreak $\uFSRAset$};
        \node [block, text width=9em, xshift=-8.5em, left of=polytope] (ccc)
            {Convex chance constraints~\cite{OnoCDC2008, LesserCDC2013} for
            \eqref{prob:xmax_compute} -- \eqref{prob:gamma_ext_W}};
        \path [line] (orig) -- (openloop) node [text centered, above, midway,
            text width=9em] {Relax state-feedback constraint} node [text centered, below, midway, text width=8em] {Proposition~\ref{prop:W_prop}\ref{prop:W_prop_OL_use}}; 
        \path [line] (openloop) -- (polytope) node [text centered, left, midway,
            text width=15.25em, xshift=0em, yshift=0.25em] {Finite subset of
            relative boundary points~\cite[Ch. 2]{BoydConvex2004}\cite[Ch. 2]{webster1994convexity}};
        \path [line] (polytope) -- (ccc) node [text centered, above, midway, text width=11em] {Risk allocation } node [text centered, below, midway, text width=11em] {Boole's inequality};
    \end{tikzpicture}}
    \caption{Underapproximative steps taken to compute the polytopic underapproximation $\uFSRAset$ of the stochastic reach set $\SRAset$ (Section~\ref{sub:GaussLTV}) via chance constraints.}\label{fig:underapprox} 
\end{figure}

Table~\ref{tab:implement} describes two approaches to approximate $W_0(
\overline{x}, \overline{U})$.  
We implement convex chance constraints via risk
allocation~\cite{LesserCDC2013, OnoCDC2008, VinodACC2019}. However, in contrast
to solving a series of linear programs~\cite{OnoCDC2008} or relying on nonlinear
solvers~\cite{LesserCDC2013}, we use piecewise affine approximation to obtain a
collection of linear constraints that conservatively enforce
the constraint $W_0( \overline{x}, \overline{U}) \geq \phi$
at significantly lower computational costs~\cite{VinodACC2019}.
This ensures that \eqref{prob:xmax_compute} and
\eqref{prob:gamma_ext_W} are linear programs,
\eqref{prob:xmax_compute_cheby} is a second order-cone program, and %
enables the use of standard conic solvers~\cite{CVX,mosek}. 
In the \ft{} approach, we initialize \texttt{MATLAB}'s
\texttt{patternsearch} using the solution obtained with the
chance constraints implementation. We use anchor points $ \overline{x}_c$
with convex chance
constraints and $\xmax$ with the \ft{} approach.  %
Figure~\ref{fig:underapprox} summarizes the conservativeness
introduced at different stages of the
convex chance constraints approach.

\section{Numerical results}
\label{sec:num}

All computations were performed on a standard laptop with an Intel
i7-4600U CPU with 4 cores, 2.1GHz clock rate and 7.5 GB RAM.
We used \texttt{SReachTools}~\cite{sreachtools} in a
\texttt{MATLAB} 2018 environment, with 
\texttt{MPT3}~\cite{MPT3}, \texttt{CVX}~\cite{CVX}, and
\texttt{MOSEK}~\cite{mosek} in the simulations. The code is
available online at
\url{https://github.com/unm-hscl/abyvinod-SRTT-2020.git}.

\subsection{Integrator chain: Interpolation \& scalability}

Consider a chain of integrators,
{
\begin{align}
    \bx_{k+1}&= \left[ {\begin{array}{ccccc} 
    1 & N_s & \frac{N_s^2}{2} & \cdots  & \frac{N_s^{n-1}}{(n-1)!}   \\   
    0 & 1   & N_s              & \cdots  &   \\  
    \vdots & &                 & \ddots  & \vdots                      \\
    0 & 0 & 0 & \cdots & 1                      \\
    \end{array} } \right]\bx_k + \left[\begin{array}{c}
    \frac{ N_s^n }{n!}\\ \frac{N_s^{n-1}}{(n-1)!} \\\vdots\\ N_s  
    \end{array}\right] u_k + \bw_k \nonumber %
\end{align}}\normalsize
with state $\bx_k\in \mathbb{R}^n$, input $u_k\in
\mathcal{U}\subset \mathbb{R}$, a Gaussian
disturbance $\bw_k \sim \mathcal{N}( \overline{0}_n,0.01I_n)$, sampling time
$N_s=0.1$, and time horizon $N$.  Here, $I_n$ refers to the $n$-dimensional
identity matrix and $ \overline{0}_n$ is the $n$-dimensional zero vector.

For the double integrator ($n=2$), we consider the stochastic viability problem \cite{AbateHSCC2007}, with
$\mathcal{T}_k={[-1,1]}^2$, $k\in \mathbb{N}_{[0,N]}$ for
$N=10$, and %
$ \mathcal{U}=[-0.1,0.1]$.  We compute
$\uFSRAset$ using Algorithm~\ref{algo:poly}
via convex
chance constraints (\texttt{SReachSetCcO} in
\texttt{SReachTools}) with $\Npoly=32$, %
and compare it to
$\SRAset$ from grid-based dynamic programming
(\texttt{SReachDynProg} in \texttt{SReachTools}), with 
grid spacing of 0.05 in the state and input spaces.

Figure~\ref{fig:intg}a shows that Algorithm~\ref{algo:poly},
which computes $\uFSRAset$, provides a good
underapproximation of the true stochastic reach set
$\SRAset$.  The advantage of using state-feedback $\pi^\ast$
over an open loop controller $\rho^\ast$ is seen in the
underapproximation ``gaps'' between the polytopes
(Proposition~\ref{prop:W_prop}\ref{prop:W_prop_OL_use}).
In Figure~\ref{fig:intg}b,
the interpolated polytopic underapproximation (from
Corollary~\ref{cor:interp} and Algorithm~\ref{algo:interp})
provides a good approximation of the true sets $\SRAsetbeta$
and $\FSRAsetbeta$ at $\beta=0.85$.

To demonstrate scalability, we consider the stochastic reach-avoid problem
with a $40$D integrator, with $N=5$,
$\mathcal{T}_k={[-10,10]}^{40}$ for $k\in \mathbb{N}_{[0,N-1]}$,
$ \mathcal{T}_N = {[-8,8]}^{40}$, and %
$ \mathcal{U}=[-1,1]$.  Dynamic
programming is clearly not feasible for comparison.
A $2$D slice of the polytopic underapproximation, at $ \overline{x}_0={[x_1\ x_2\ 0\ 0\ \cdots\ 0]}^\top\in
\mathbb{R}^{40}$, is shown in 
Figure~\ref{fig:intg}c for $\Npoly=8$, computed via Algorithm~\ref{algo:poly}
with convex chance constraints.  As shown in
Figure~\ref{fig:intg}d, that the underapproximative
interpolation obtained from Algorithm~\ref{algo:interp}
appears to be tight; the high ratio ($0.994$) of the volume
of interpolated polytopic underapproximation to the volume
of $\uFSRAset$ confirms this.

\begin{table}
    \caption{Computation time (in seconds) for a chain of integrators.}
    \label{tab:computeTime}
    \setlength\tabcolsep{2pt}
    \centering
    \begin{tabular}{c|c|c|c||c|c|c|}
        \cline{2-7}									%
        & \multicolumn{3}{c||}{$n=2$ ($\Npoly = 32$)}& \multicolumn{3}{c|}{$n=40$ ($\Npoly=8$)}\\\hline
    \multicolumn{1}{|c|}{$\alpha$} & $0.6$  & $0.85$ & $0.9$ &$0.6$  & $0.85$ & $0.9$ \\\hline\hline
    \multicolumn{1}{|c||}{Algorithm~\ref{algo:poly}} & 
    $11.86$ & $10.12$ &  $9.80$ &
    $921.29$ & $842.29$ & $781.78$ \\\hline
    \multicolumn{1}{|c||}{Algorithm~\ref{algo:interp}} & --
                                                       &
    $0.006$ & --  & -- & $0.013$ & -- \\\hline
    \multicolumn{1}{|c||}{Dyn. prog.} & \multicolumn{3}{c||}{
    $5.94$} & \multicolumn{3}{c|}{\multirow{2}{*}{Not possible}} \\\cline{1-4}
    \multicolumn{1}{|c||}{Theorem~\ref{thm:interp}} & -- &
    $0.035$ & --  & \multicolumn{3}{c|}{} \\\hline
   \end{tabular}
\end{table}

As expected, Algorithm~\ref{algo:poly} outperforms dynamic programming in computation time 
(Table~\ref{tab:computeTime}).
This is a direct
consequence of the convexity, compactness, and
underapproximative properties established in
Section~\ref{sec:OL}. The interpolation scheme provided in 
Algorithm~\ref{algo:interp}
is faster than Algorithm~\ref{algo:poly} 
by three orders of magnitude in $2$D,
and by four
orders of magnitude in $40$D.

\begin{figure*}
    \centering
    \newcommand{\trimValuesIntg}{445 80 490 10}
    \includegraphics[Trim=\trimValuesIntg, clip, width=0.24\linewidth]{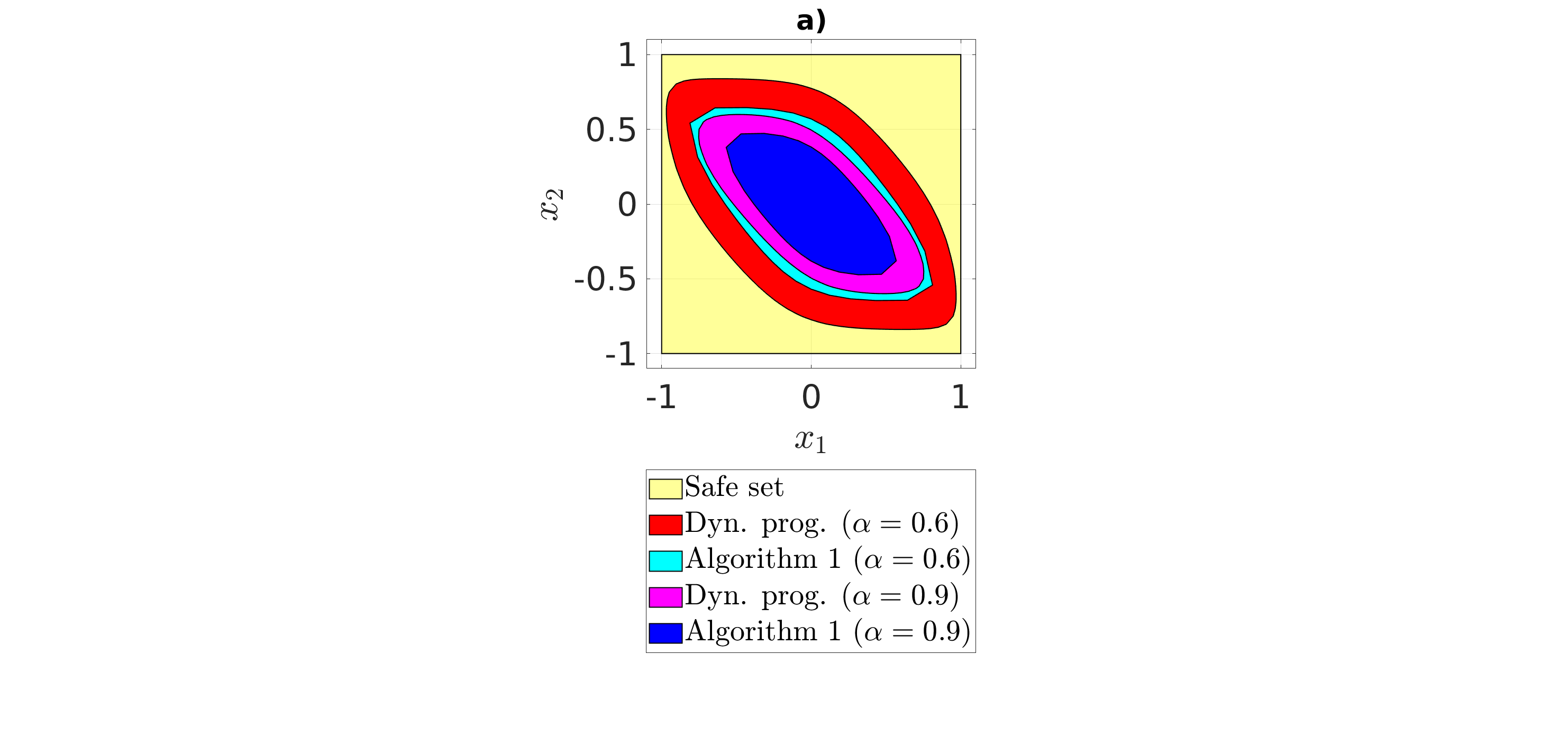}
    \includegraphics[Trim=\trimValuesIntg, clip, width=0.24\linewidth]{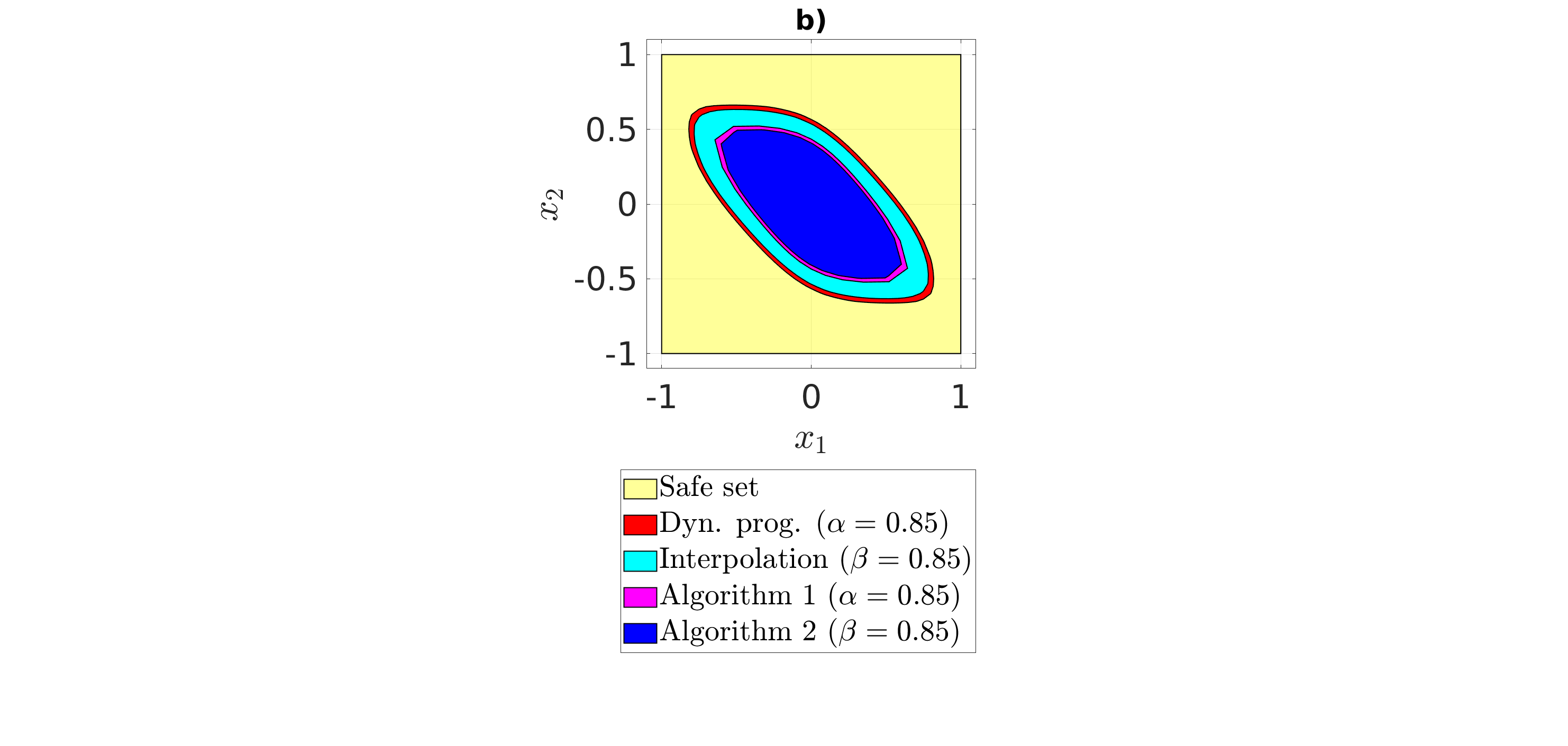}
    \includegraphics[Trim=\trimValuesIntg, clip, width=0.24\linewidth]{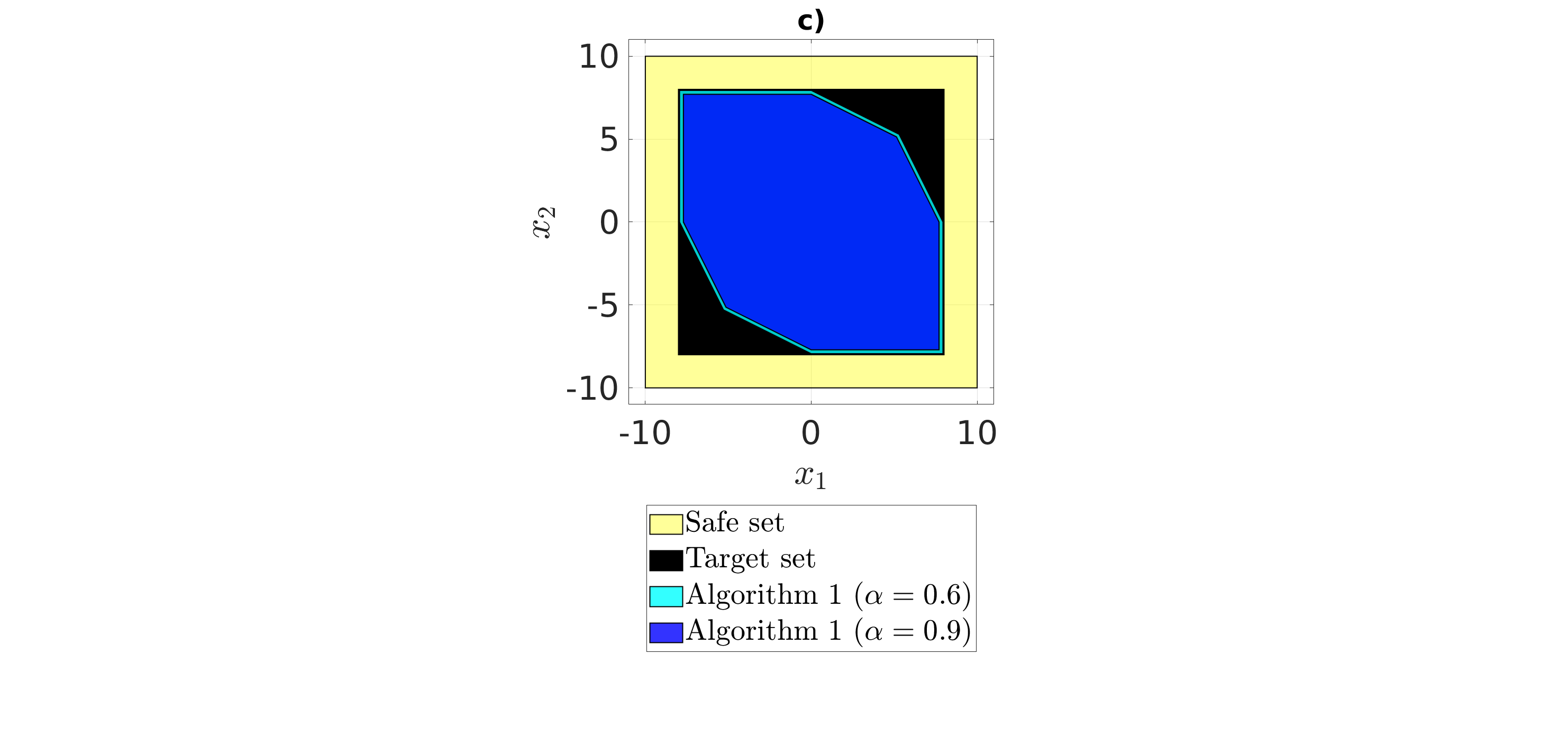}
    \includegraphics[Trim=\trimValuesIntg, clip, width=0.24\linewidth]{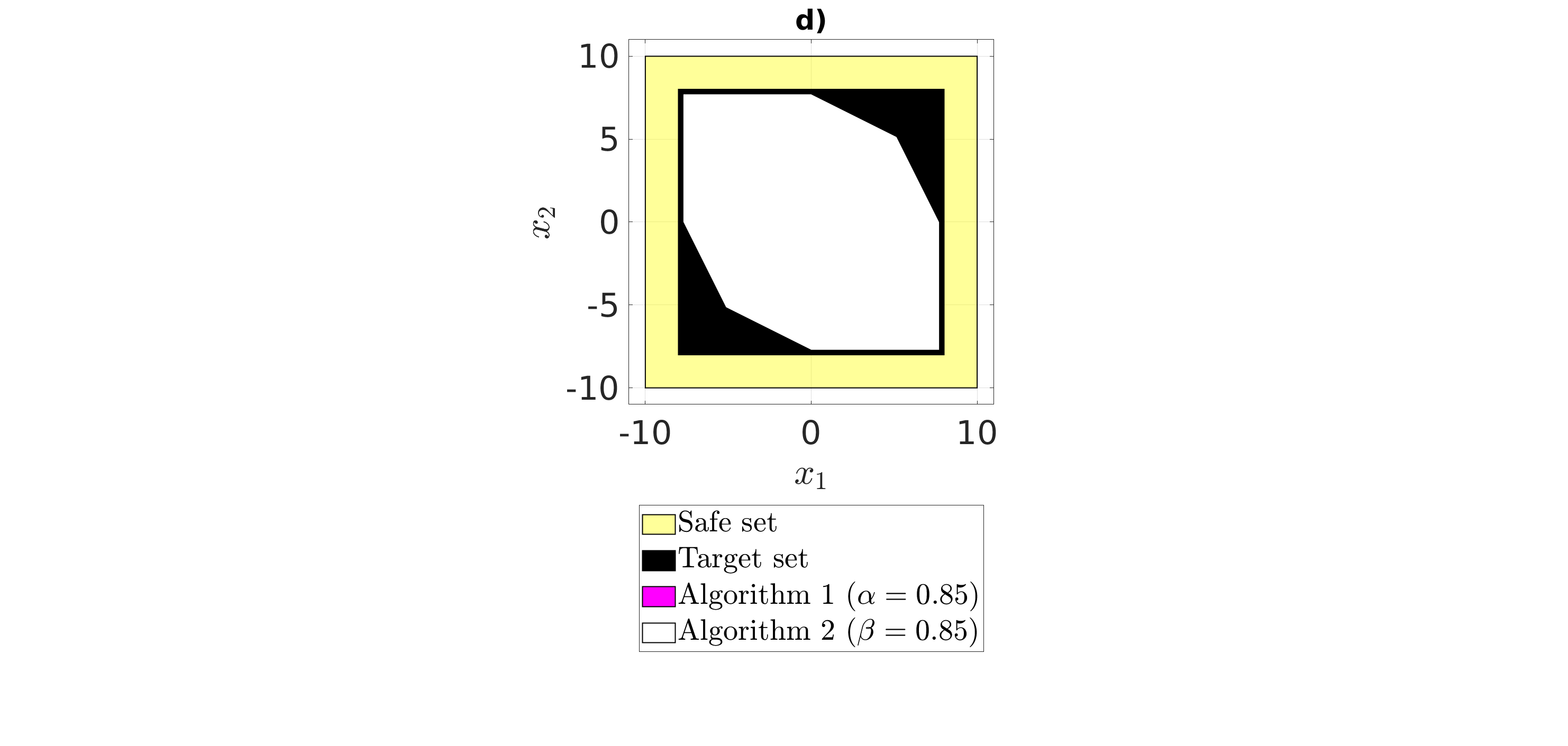}\\
    \caption{Stochastic reachable sets for a chain of integrators. 
        Figures a) and
        c) show the stochastic reach sets and their
        underapproximations for $\alpha\in\{0.6,0.9\}$ for $n=2$ and $n=40$,
        respectively. Figures b) and d) show the underapproximative
        interpolations for $n=2$ and $n=40$, respectively. For $n=40$, the plots
        show the initial state $ \overline{x}_0 =[x_1\ x_2\ 0\ \cdots\
    0]^\top$.  We choose $ \overline{x}_c$ \eqref{prob:xmax_compute_cheby} for $
\overline{x}_\mathrm{anchor}$.}\label{fig:intg} 
\end{figure*}

\begin{figure}
    \centering
    \includegraphics[width=0.47\linewidth]{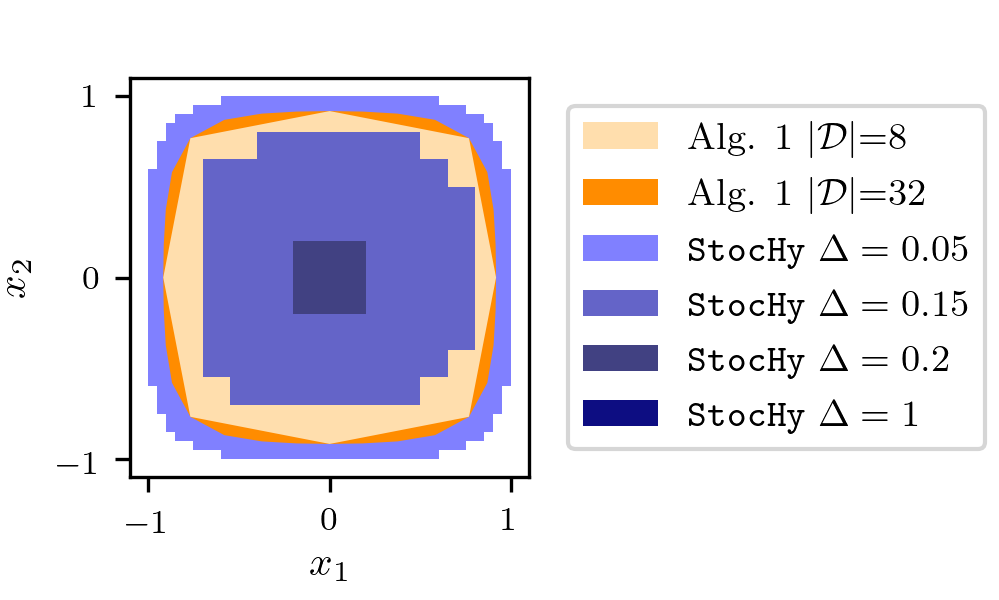}\ 
    \includegraphics[width=0.42\linewidth]{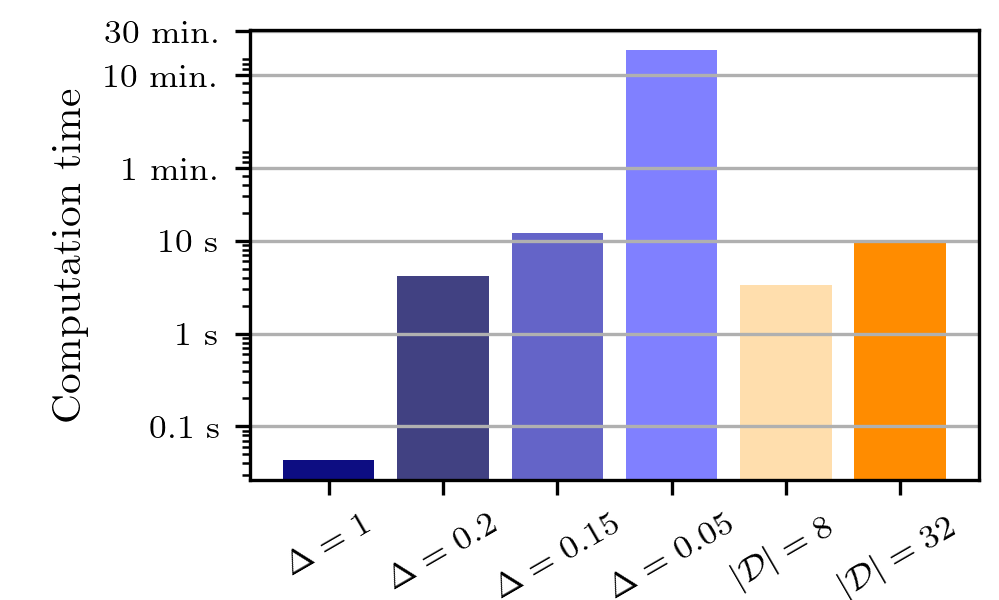} 
    \caption{Stochastic reach sets for $\alpha=0.6$ and
        $n=2$.  (Top) 
        Algorithm~\ref{algo:poly} produced a larger underapproximation than \texttt{StocHy}
        for $\Delta \in \{1, 0.2, 0.15\}$, but a smaller underapproximation 
        for $\Delta = 0.05$. 
        \texttt{StocHy}
    returned an empty set for $\Delta=1$. (Bottom)
As expected, computational effort of
Algorithm~\ref{algo:poly} and \texttt{StocHy} increases with higher fidelity.}
\label{fig:stochy_reach_sets_2D}
\end{figure}

\begin{figure*}
    \centering
    \newcommand{\trimValuesComputeTime}{0 15 0 100}
    \includegraphics[Trim=\trimValuesComputeTime, clip, width=0.45\linewidth]{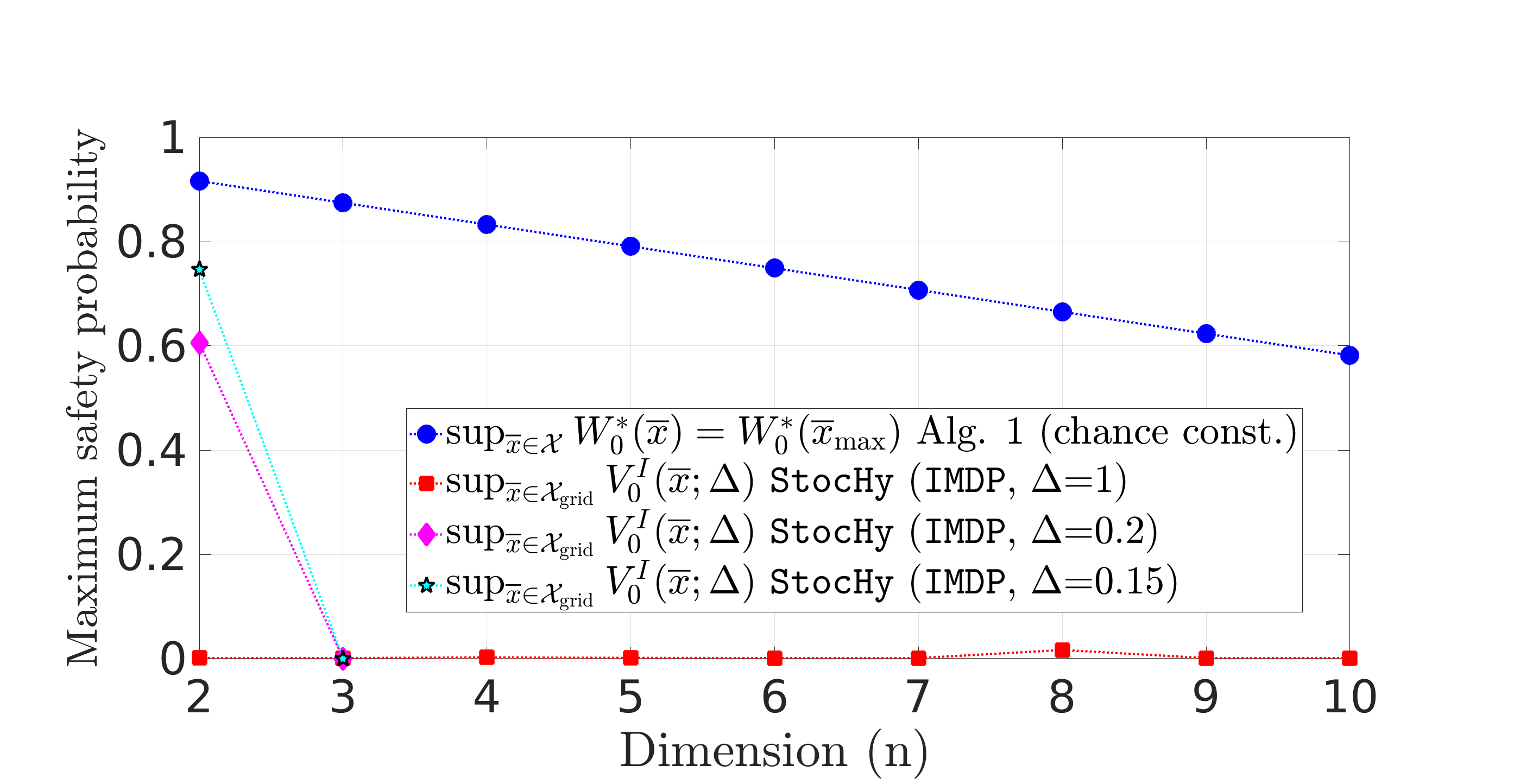}\hskip1em
    \includegraphics[Trim=\trimValuesComputeTime, clip, width=0.45\linewidth]{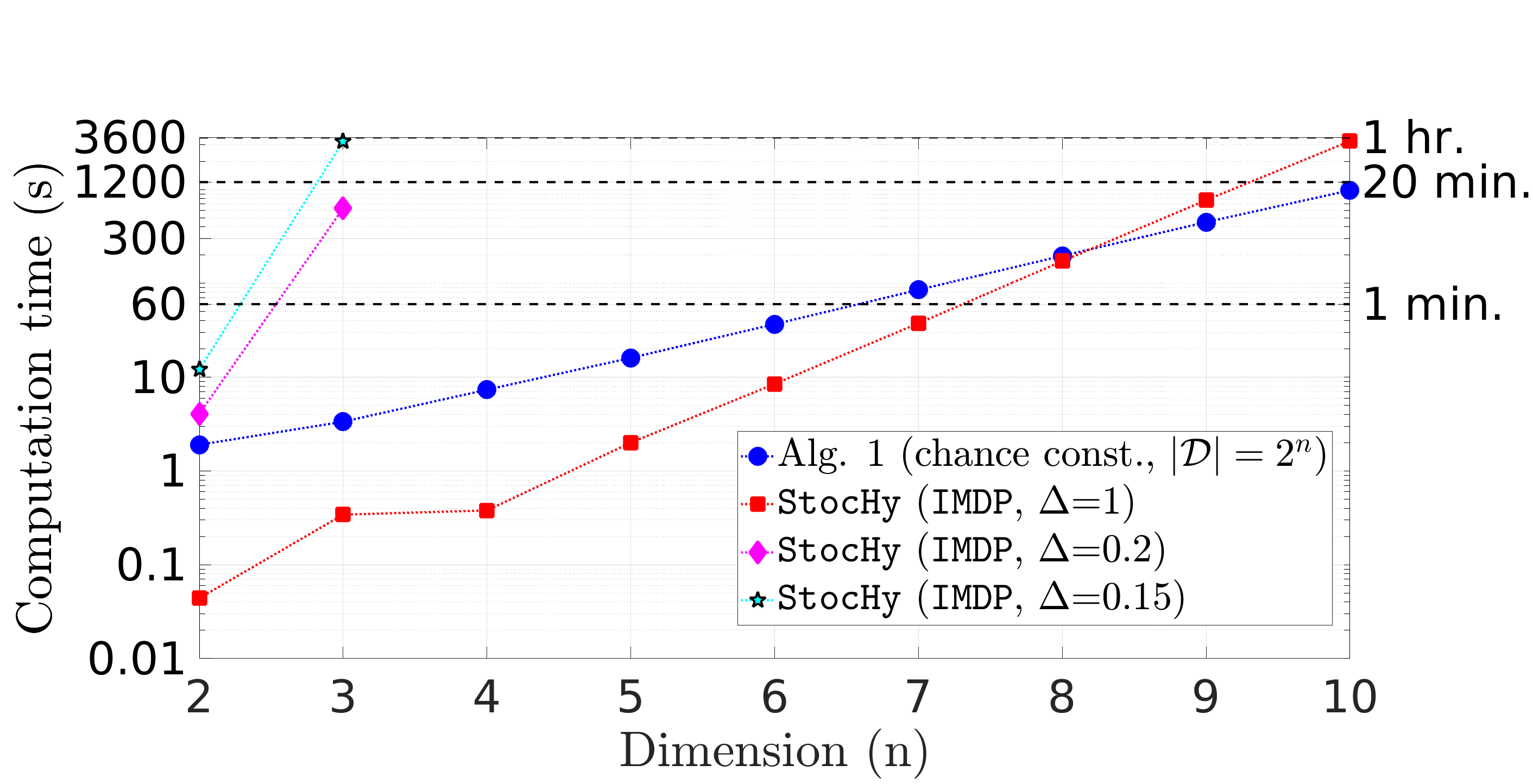}
    \caption{Comparison of Algorithm~\ref{algo:poly} with
        \texttt{StocHy} for various $n$. 
        (Left) Algorithm~\ref{algo:poly} generates higher maximal safety probabilities than \texttt{StocHy};
        for $n\geq 3$, \texttt{StocHy} returns trivial underapproximations, despite reductions in 
        grid step size. (Right) Computation of
        Algorithm~\ref{algo:poly}
        with $|\mathcal D| = 2^n$ scales significantly
        better than \texttt{StocHy} for fine grids,
        and comparably for coarse grids.}
\label{fig:SReachToolsVsStocHy}
\end{figure*}

\begin{figure*}
    \centering
    \newcommand{\trimValuesComputeTime}{0 10 0 30}
    \includegraphics[Trim=\trimValuesComputeTime, clip,
    width=0.45\linewidth]{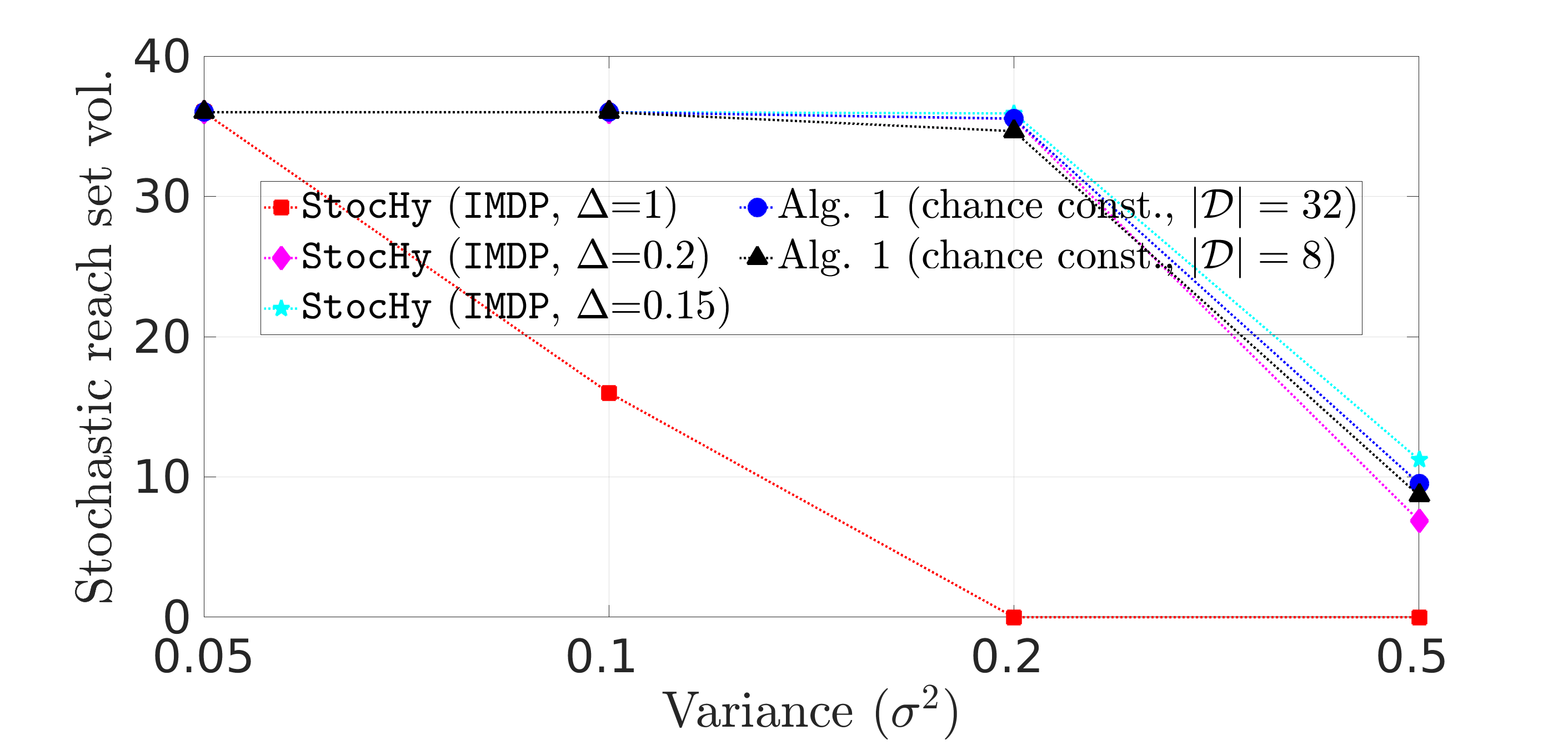}\hskip1em 
    \includegraphics[Trim=\trimValuesComputeTime, clip,
    width=0.45\linewidth]{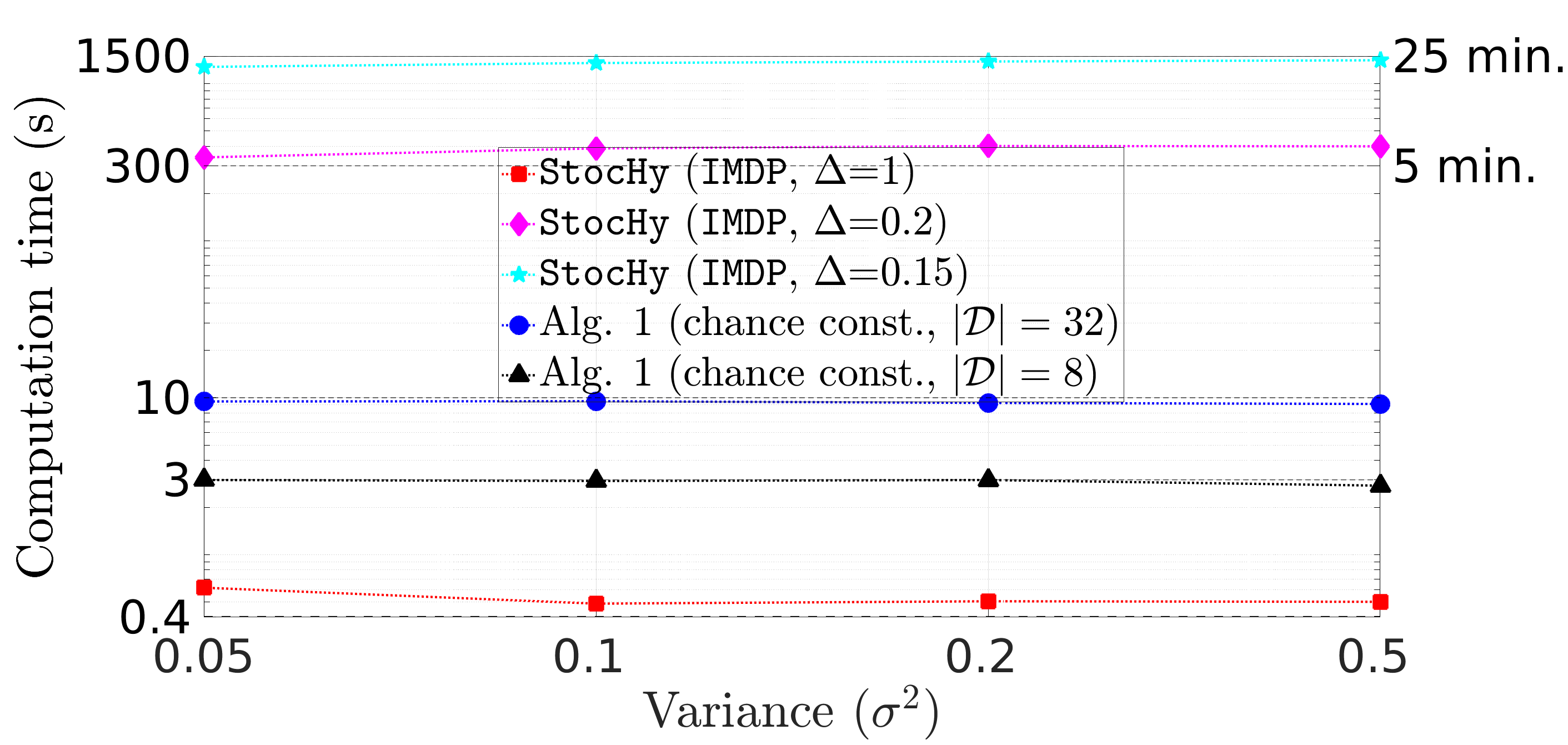}
    \caption{Comparison of Algorithm~\ref{algo:poly} with
        \texttt{StocHy} for various $\sigma^2$ and $n=2$. 
        (Left) Algorithm~\ref{algo:poly} exhibits a lower degree of
        conservativeness (larger volume) than
        \texttt{StocHy}, except for a fine grid
        $\Delta=0.15$.
        (Right) The computational effort required by
    Algorithm~\ref{algo:poly}  is two orders of magnitude ($3-10$ seconds as compared to $300-1500$ seconds) lower than
\texttt{StocHy} for $\Delta\in\{0.15, 0.2\}$. 
}
\label{fig:SReachToolsVsStocHy_various_var}
\end{figure*}

\subsection{Comparison of Algorithm~\ref{algo:poly} and \texttt{StocHy}}
\label{sub:stochy}

\texttt{StocHy} solves stochastic viability or reach-avoid problems
for stochastic hybrid systems~\cite{cauchi2019stochy,cauchi2019efficiency} via
abstraction and dynamic programming. \texttt{StocHy} computes
look-up
tables $V^I_0(x;\Delta):\mathcal{X}_\mathrm{grid}\to [0,1]$
defined over a grid $\mathcal{X}_\mathrm{grid}\subset
\mathcal{X}$ with step size $\Delta > 0$.  We focus on the
\texttt{IMDP} implementation of StocHy, since its solution
$V^I_0(x;\Delta)$ lower bounds %
the true safety probability \eqref{eq:ProbHatR}.  
The alternative implementation,
\texttt{FAUST$^2$} \cite{soudjani2015faust} scales more
poorly and is slower than
\texttt{IMDP}~\cite{cauchi2019efficiency}. 
Table~\ref{tab:differences} summarizes the
differences between \texttt{SReachTools} (which implements
Algorithm~\ref{algo:poly}) and \texttt{StocHy}.
Because
\texttt{IMDP} cannot accommodate continuous control input
and time-varying target sets,
we consider the uncontrolled system
\begin{align}
    \bx_{k+1}&= 0.8\bx_k + \bw_k\label{eq:LTI_stochy}
\end{align}
with state $\bx_k \in
\mathbb{R}^n$ and disturbance $\bw_k\sim \mathcal{N}(0,
\sigma^2 I_n)$, and compute the $0.6$-stochastic viability
set for $N=10$ and $\mathcal{T}_k=\mathcal{S}$. 

\newcommand{\propertymultirow}[1]{
\begin{minipage}{3cm}
            \centering
            \vspace*{0.2em}#1
    \end{minipage}
}
\newcommand{\toolmultirow}[1]{
\begin{minipage}{2.75cm}
            \centering
            \vspace*{0.2em}#1
    \end{minipage}
}
\begin{table}
    \centering
    \caption{Capabilities of \texttt{SReachTools}
   and \texttt{StocHy}.}
    \begin{tabular}{|c||c|c|} 
        \hline									 
        Property & \texttt{SReachTools} & \texttt{StocHy} \\
        \hline\hline

        System model & Linear time-varying & \toolmultirow{Hybrid time-
        invariant}  \\[1.5ex]\hline

        Specification & \toolmultirow{Reachability of \\
        target tube \eqref{prob:MP}} &
        \toolmultirow{Viability or reach-avoid}\\[1.5ex]
        \hline

        \propertymultirow{Underapproximative \\
        verification} & \checkmark &
        \checkmark \\[1.5ex] \hline

        \propertymultirow{Underapproximative \\controller
        synthesis }  & \checkmark & \\[1.5ex] \hline 
        
        Controller type  & 
        \toolmultirow{Open-loop or affine feedback} &
        \toolmultirow{Grid based\\ state feedback} \\[1.5ex] \hline 

   \end{tabular}
   \label{tab:differences}
\end{table}

\emph{Advantage of grid-free approach}: To compare the
stochastic reach sets produced by Algorithm~\ref{algo:poly}
(grid-free)
to the sets produced by \texttt{StocHy} (grid-based),
we chose
$\sigma^2=0.05$, $n=2$, $
\mathcal{S}=[-1,1]^2$, and varied the
grid step size $\Delta\in\{0.05,0.15,0.2, 1\}$. Figure
\ref{fig:stochy_reach_sets_2D} shows that the sets generated
by Algorithm~\ref{algo:poly} are less conservative (larger
volume) than \texttt{StocHy}, except when a very fine grid
($\Delta = 0.05$) is used. However, \texttt{StocHy} was two
orders of magnitude slower than Algorithm~\ref{algo:poly}
for $\Delta=0.05$. 

\emph{Scalability evaluation}: We
compared the maximal safety probabilities,
$\sup_{\overline{x}\in
    \mathcal{X}} W_0^\ast( \overline{x})$ via Algorithm~\ref{algo:poly},
    and $\sup_{\overline{x}\in \mathcal{X}_\mathrm{grid}}
    V_0^I( \overline{x};\Delta)$ via \texttt{StocHy}, with 
$\sigma^2=0.05$ and $
\mathcal{S}=[-1,1]^n$.
Figure \ref{fig:SReachToolsVsStocHy} (left) shows that 
\texttt{StocHy} is more conservative, as it consistently computes a lower
maximal safety probability than Algorithm~\ref{algo:poly}, 
despite reductions in grid size ($\Delta\in\{0.15,0.2\}$).
Indeed, \texttt{StocHy} generated trivial
underapproximations for $n \geq 3$, and for $n\geq
4$, the computational cost for $\Delta\in\{0.15,0.2\}$ was
prohibitive 
(Figure~\ref{fig:SReachToolsVsStocHy} (right)).
Computationally, Algorithm~\ref{algo:poly} scales comparably to
\texttt{StocHy} for a coarse grid ($\Delta=1$), and
significantly better than \texttt{StocHy} for finer grids
($\Delta\in\{0.15,0.2\}$). We used chance constraint approach with $\Npoly=2^n$ and $\xmax$ as
$ \overline{x}_c$ in Algorithm~\ref{algo:poly}.

\emph{Effect of disturbance variance}:
We evaluate the volume of the stochastic reach set 
for 
$\sigma^2\in\{0.05, 0.1, 0.2,
0.5\}$ for $n=2$ and $ \mathcal{S}=[-3,3]^2$, 
via Algorithm~\ref{algo:poly} and \texttt{StocHy}, 
in Figure~\ref{fig:SReachToolsVsStocHy_various_var}. 
As expected, the reach set volume decreases as the
disturbance variance increases.
Algorithm~\ref{algo:poly} produces stochastic reach sets of
similar volume as that of
\texttt{StocHy} at finer grids $\Delta\in\{0.15,0.2\}$,
in significantly lower computation time.
\texttt{StocHy} is faster than Algorithm~\ref{algo:poly} for $\Delta=1$, but is significantly more conservative. %

In summary, we observe empirically that 
Algorithm~\ref{algo:poly} tends to be less
conservative and computationally faster than
\texttt{StocHy}. 

\subsection{Spacecraft rendezvous}
\label{sub:CWH}

We consider two spacecraft in the same elliptical orbit.
One spacecraft, referred to as the deputy, must approach and
dock with another spacecraft, referred to as the chief,
while remaining in a line-of-sight cone, in which accurate
sensing of the other vehicle is possible.  The relative
dynamics are described by the Clohessy-Wiltshire-Hill (CWH)
equations \cite{wiesel1989_spaceflight} with an additive
stochastic noise to account for model uncertainties,
\begin{align}
    \ddot{x} - 3 \omega x - 2 \omega \dot{y} = m_{d}^{-1}F_{x},\quad\ddot{y} + 2 \omega \dot{x} = m_{d}^{-1}F_{y}.
  \label{eq:2d-cwh}
\end{align}
The chief is located at the origin, the position of the deputy is $x,y \in
\mathbb{R}$, $\omega = \sqrt{\Gamma/R_{0}^{3}}$ is the orbital frequency,
$\Gamma$ is the standard gravitational parameter, and
$R_{0}$ is the orbital radius of the spacecraft.
See~\cite{LesserCDC2013, GleasonCDC2017} for further
details.

We define the state as $z = [x,y,\dot{x},\dot{y}] \in \mathbb{R}^{4}$ and input
as $u = [F_{x},F_{y}] \in \mathcal{U} =
[-u_M,u_M]^2\subseteq\mathbb{R}^{2}$ for some $u_M > 0$. We discretize the dynamics (\ref{eq:2d-cwh})
in time to obtain %
  $$\overline{z}_{k+1} = A \overline{z}_{k} + B \overline{u}_{k} +
  \bw_{k},$$
with a Gaussian %
disturbance $w_{k} \in \mathbb{R}^{4}$ with zero mean, and
covariance $10^{-4}\times\mbox{diag}(1, 1, 5 \times 10^{-4},
5 \times 10^{-4})$.  For a time horizon of $N = 5$, we
define the target tube
\begin{align}
    \mathcal{T}_5 &= \left\{ \overline{z}\in \mathbb{R}^4: |z_{1}| \leq 0.1, -0.1 \leq z_{2} \leq 0, |z_{3}| \leq 0.01, |z_{4}| \leq 0.01 \right\},\mbox{ and } \nonumber \\
    \mathcal{T}_k &= \left\{ \overline{z}\in \mathbb{R}^4: 2 \leq |z_{1}| \leq -z_{2},
    |z_{3}|\leq 0.5, |z_{4}| \leq 0.5 \right\}\nonumber
\end{align}
for $k\in \mathbb{N}_{[0,4]}$.
We want to solve the stochastic reach-avoid problem for
$\alpha=0.8$ with initial velocity
$\dot{x}=\dot{y}=v_\mathrm{const}$ km/s for two scenarios: P1) 
$v_\mathrm{const}=0$ and $u_M = 0.1$, and P2)
$v_\mathrm{const}=0.01$ and $u_M = 0.01$.

We construct underapproximative stochastic reach sets using
Algorithm~\ref{algo:poly} with convex chance constraint and
\ft{} approaches. We compare these results with the
grid-based approach proposed in~\cite{LesserCDC2013}.
In~\cite{LesserCDC2013}, the problem 
$\sup_{ \overline{U} \in \mathcal{U}^N} W_0( \overline{x},
\overline{U})$ is solved in a chance-constrained formulation
using \texttt{MATLAB}'s \texttt{fmincon} for $
\overline{x}\in \mathcal{X}_\mathrm{grid}$ (over the visible
area) with grid step-size $\Delta=0.01$. 
Figure~\ref{fig:CWH} shows the
sets, and Table~\ref{tab:CWHandDubinsTime} shows the computation
time, and the statistics of the approximation error of the
open-loop maximal reach probability using Monte Carlo
simulations with $10^5$ scenarios at the vertices of
$\uFSRAset$. Due to the dimensionality, this problem is
intractable via dynamic programming. This problem is also
hard for abstraction-based techniques like
\texttt{StocHy}~\cite{cauchi2019stochy}, since $
\mathcal{T}_k$ are not axis-aligned hypercubes.  

\begin{figure}
    \centering
    \newcommand{\trimValuesCWHExampleA}{20 90 110 130}
    \includegraphics[height=0.15\textheight, Trim=\trimValuesCWHExampleA,
    clip]{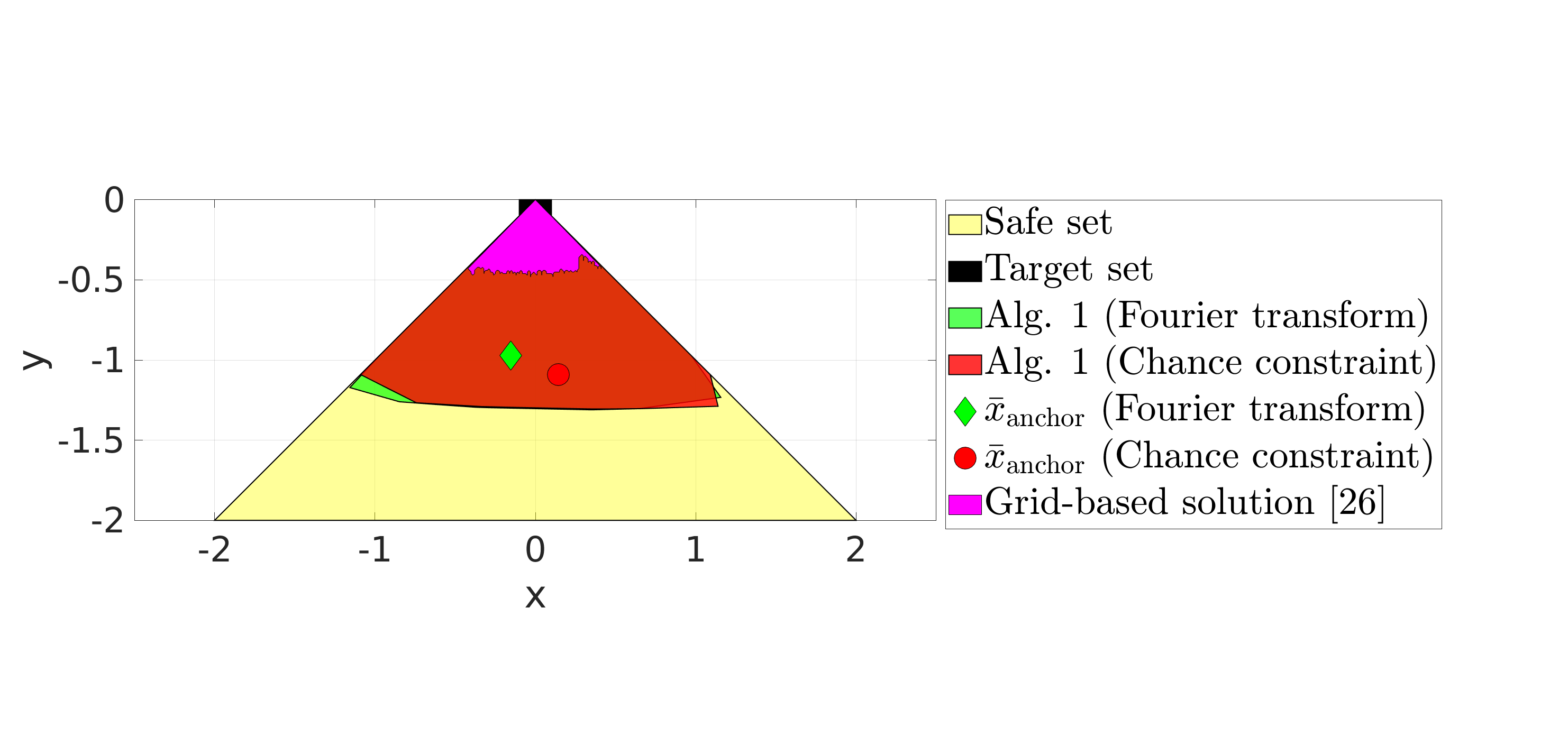}
    \newcommand{\trimValuesCWHExampleB}{160 10 730 30}
    \includegraphics[height=0.18\textheight, Trim=\trimValuesCWHExampleB,
    clip]{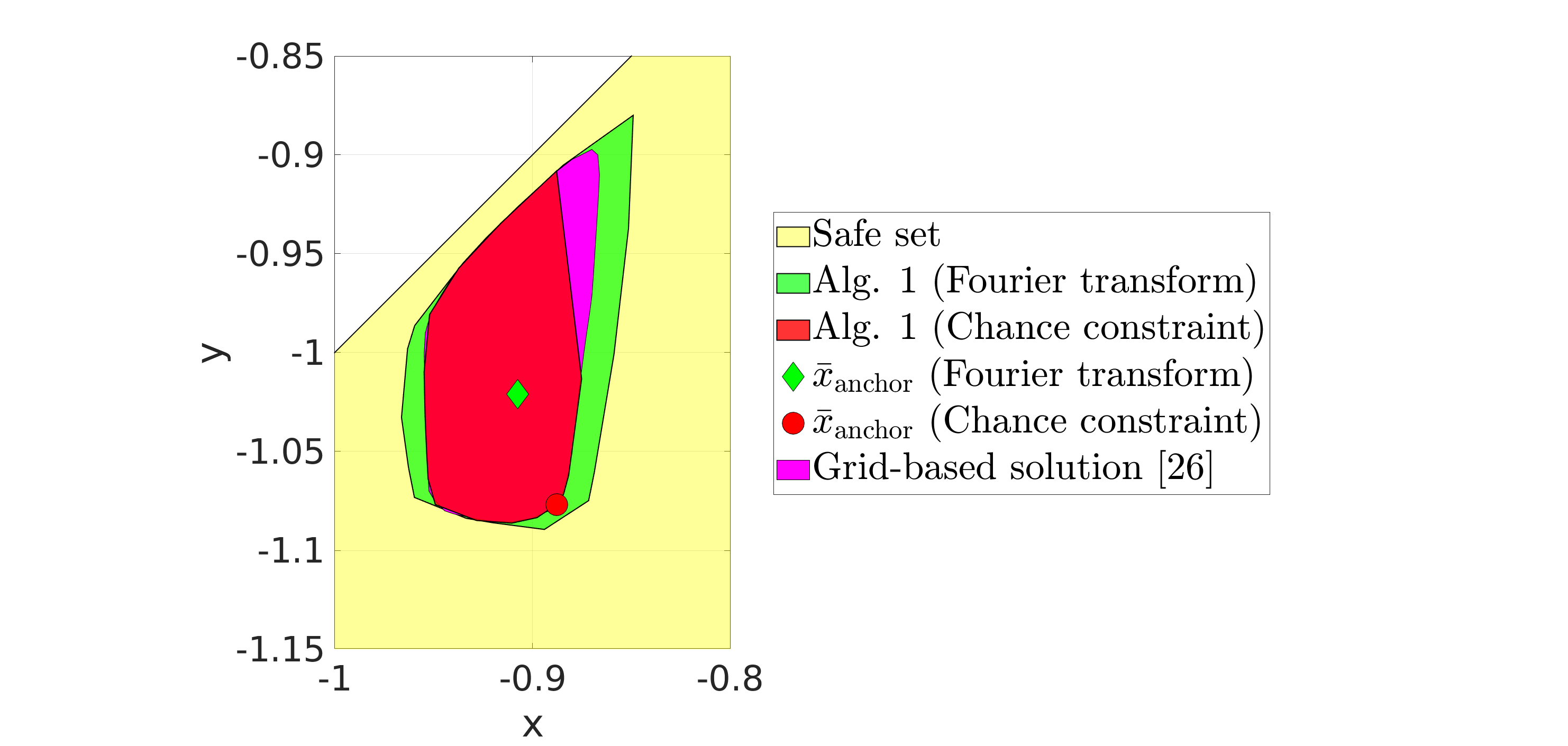}
    \caption{Spacecraft rendezvous and docking problem.
        Comparison of Algorithm~\ref{algo:poly} using chance constraints
        and \ft{} approaches and the grid-based
        solution~\cite{LesserCDC2013}.  Initial velocity
        is zero for case P1 (top), and non-zero for case P2
        (bottom).  We choose $\Npoly=32$ and $
    \overline{x}_c$ \eqref{prob:xmax_compute_cheby} for $
\overline{x}_\mathrm{anchor}$.}\label{fig:CWH} 
\end{figure}

\begin{table}
   \caption{Computation time (seconds) and
       approximation error in
       $W_0^\ast(\xanchor+\theta_i^\ast \overline{d}_i)$ for
       each $i\in \mathcal{D}$.  The approximation 
       error is the difference of the open-loop reach
       probability estimated using Monte Carlo simulation
       ($10^5$ particles) and the optimal solution of
       \eqref{prob:gamma_ext_W} at each direction vector
       $\overline{d}_i$.  
   Only the visible area in \cite{LesserCDC2013} is gridded for P2, 
   to avoid time-out ($<3$ hours).}
   \centering
    \begin{tabular}{|c||c|c|c|c|c|c|c|}
        \hline
        \multirow{3}{*}{Fig.}& \multicolumn{3}{c|}{Chance-constraint} &
        \multicolumn{3}{c|}{\ft{}} & ~\cite{LesserCDC2013} \\ \cline{2-8}
        & Time & \multicolumn{2}{c|}{Error at vertex} & Time  &
        \multicolumn{2}{c|}{Error at vertex} & Time \\\cline{3-4}\cline{6-7}
                                             & (s)  & mean &
        std. &  (s)    & mean & std. & (s) \\\hline\hline
        \ref{fig:CWH}, top     &  $9.84$  & $0.037$ & $0.028$
                               & $679.93$ & $0.006$ & $0.007$ & $8369.84$\\ \hline %
        \ref{fig:CWH}, bottom  & $10.98 $ & $0.017$ & $0.004$ 
                               & $842.77$ & $0.015$ & $0.006$ & $764.92$\\ \hline
                               \multicolumn{4}{c}{}\\[-2ex]\cline{1-4}
        \ref{fig:Dubins} &  $32.19$ &  $0.128$  & $0.036$ &
        \multicolumn{4}{c}{}\\ \cline{1-4}
    \end{tabular}
   \label{tab:CWHandDubinsTime}
\end{table}

Algorithm~\ref{algo:poly} implemented using convex chance
constraint approach (\texttt{SReachSetCcO} in
\texttt{SReachTools}) is significantly faster (two orders of
magnitude) than the \ft{} approach (\texttt{SReachSetGpO} in
\texttt{SReachTools}), since it relies on a
standard conic solver (like \texttt{MOSEK}~\cite{mosek} or
\texttt{CVX}) instead of a
gradient-free nonlinear optimization solver
(\texttt{MATLAB}'s \texttt{patternsearch}).  On the other
hand, the \ft{} approach typically provides a larger
underapproximative reach set.  This is because it
approximates $W_0(\cdot)$ directly, instead of relying using
Boole's inequality for an underapproximation, and we
initialize \texttt{MATLAB}'s \texttt{patternsearch} with the
solution obtained from the convex chance constraint
implementation. 
Algorithm~\ref{algo:poly} with chance
constraints is two to four orders of magnitude faster than~\cite{LesserCDC2013}, since
it relies on convex solvers and does not require gridding, and is more computationally robust.
We also
used Monte-Carlo simulations from each of the vertices of
the computed polytopes to validate that the safety
prescriptions of Algorithm~\ref{algo:poly} implemented using
convex chance constraints or \ft{} is
underapproximative (approximation error in
Table~\ref{tab:CWHandDubinsTime} has a positive mean).

\subsection{Dubin's vehicle: Linear time-varying system with time-varying target}
\label{sub:Dubins}

We consider the problem of driving a Dubin's vehicle under a
known turning rate sequence while staying within a target
tube.  When the turning rate sequence and the initial
heading are known, the resulting sequence of heading angles
can be constructed \emph{a priori}.  The corresponding
linear time-varying vehicle dynamics are, 
\begin{align}
    \bx_{k+1}&=\bx_{k} + %
    \left[\begin{array}{c}
            T_s\cos(\varphi_0+ T_s\sum_{j=0}^{k-1} \omega_j) \\         
            T_s\sin(\varphi_0+ T_s\sum_{j=0}^{k-1} \omega_j) \\         
\end{array}\right] u_k +\bfeta_k\label{eq:LTV_dubins}
\end{align}
with position $\bx_k\in \mathbb{R}^2$, heading velocity $u_k\in [0, u_{M}]$,
sampling time $T_s$, known initial heading $\varphi_0$, time
horizon $N$, known sequence of turning rates
$\{\omega_k\}_{k=0}^{N-1}$, and a Gaussian random process
$\bfeta_k\sim \mathcal{N}(\overline{\mu}_{\bfeta}, \Sigma_{\bfeta})$.
For some fraction $\delta\in (0,1]$, let ${\{ \overline{c}_k\}}_{k=0}^{N-1}$ be
the resulting \emph{nominal} trajectory of \eqref{eq:LTV_dubins}
associated with a fixed heading velocity $ u_k = u_M\delta,\ \forall k\in
\mathbb{N}_{[0,N]}$.

We choose the parameters $N=50$, $T_s=0.1$, $\omega_k =
0.2\pi\ \forall k\in \mathbb{N}_{[0,N-1]}$, $\varphi_0 =
0.1\pi$, $u_M = 10$, $\delta=0.7$, $
\overline{\mu}_{\bfeta} = \overline{0}_2$, and
$\Sigma_{\bfeta} = 10^{-3} I_2$. 
We wish to compute the
$0.8$-level stochastic reach set of the target tube $
\mathcal{T}_k = \mathrm{Box}\left(\overline{c}_k, 4
\exp\left({\frac{-k}{N_c}}\right)\right),\ \forall k\in
\mathbb{N}_{[0,N]}$ where $N_c = 100$ is the decay time
constant and $ \mathrm{Box}( \overline{c}_k, a)\subset
\mathbb{R}^2$ is an axis-aligned box centered around $
\overline{c}_k$ with side $2a,\ a>0$. 
\texttt{StocHy}~\cite{cauchi2019stochy} does not
support time-varying dynamics. 

Figure~\ref{fig:Dubins} shows the underapproximative open-loop stochastic
reach set computed by Algorithm~\ref{algo:poly} (convex chance constrained
approach). The computed open-loop controller contains the spread of the simulated trajectories within the target tube, as shown for one of the vertices of $\uFSRAset$. Due to the long time horizon and the large state space, grid-based dynamic programming
is infeasible. Table~\ref{tab:CWHandDubinsTime} shows that
the conservatism (approximation error has a positive and
large mean) in the computed lower bounds at the vertices,
due to Boole's inequality.

\begin{figure}
    \centering
    \newcommand{\trimvaluesDubins}{225 10 270 40}
    \includegraphics[width=0.45\linewidth,Trim=\trimvaluesDubins,clip]{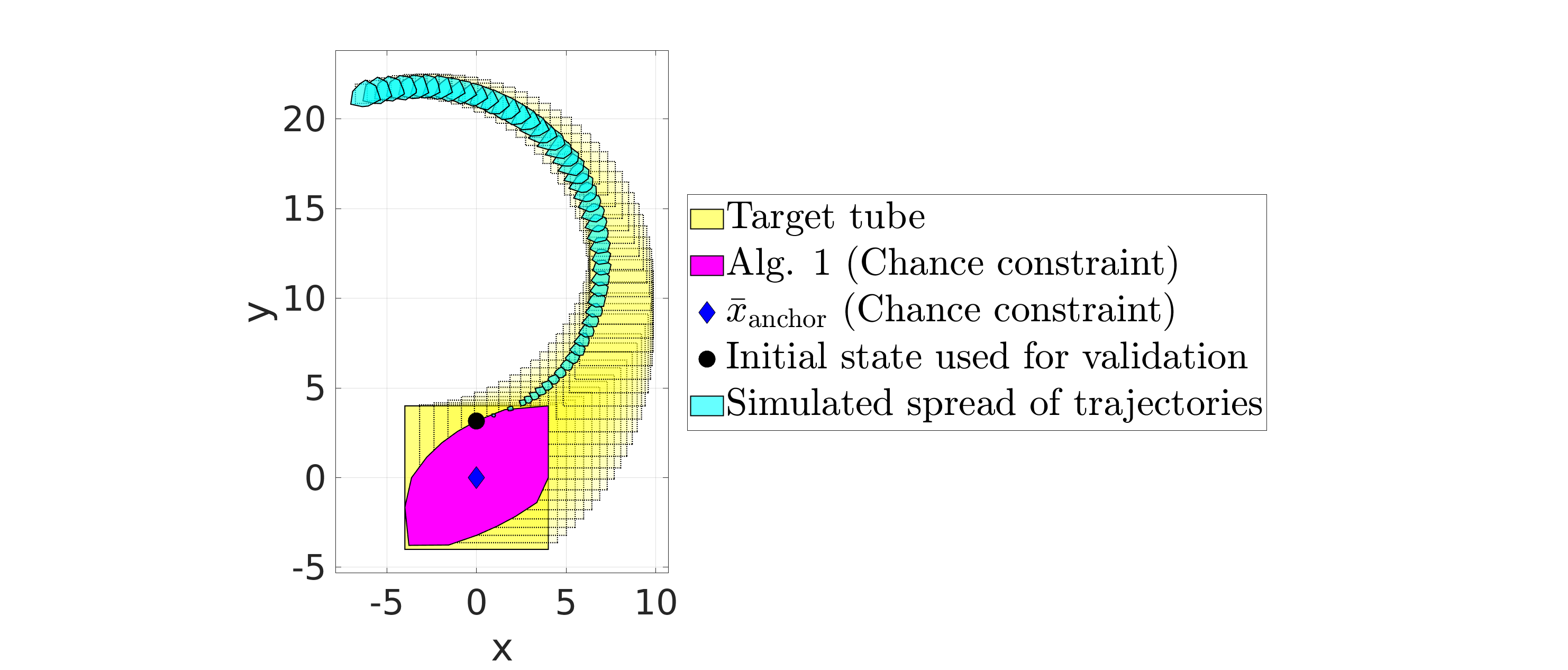}
    \caption{
        Stochastic reach set of a Dubin's vehicle with a target tube. The cyan
        polytopes show the spread of the trajectories ($10^5$ scenarios) in
        Monte Carlo simulation, under the open-loop
    controller from Algorithm~\ref{algo:poly}. We choose
    $\Npoly=16$ and $ \overline{x}_c$
\eqref{prob:xmax_compute_cheby} for $
\overline{x}_\mathrm{anchor}$.}
    \label{fig:Dubins}
\end{figure}

\section{Conclusions}
\label{sec:conc}

In this paper, we have characterized the properties of the
stochastic reachability problem of a target tube, that
guarantee existence and closed, compact, and convex
stochastic reach set.  By establishing convexity properties,
we constructed an underapproximation of the stochastic reach
set via interpolation.  Finally, we proposed a scalable,
grid-free, and anytime algorithm that provides an open-loop
controller-based polytopic underapproximation of the
stochastic reach set.

\section*{Acknowledgements}

We thank Dr. Nathalie Cauchi for her help in
setting up the comparison 
 with \texttt{StocHy}.

\newcommand{\myqed}{\hspace*{\fill}{$\blacksquare$}}
\makeatletter
\appendix
\section{Proofs for
Propositions~\ref{prop:contains},~\ref{prop:Borel},~\ref{prop:usc},
and~\ref{prop:W_prop}}
\label{app:proof}
\makeatother

\subsection{Proof of Proposition~\ref{prop:contains}}
\label{app:proof_contains}
    For any $k\in \mathbb{N}_{[0,N]}$ and $ \overline{x}\in \SRAsetk$, $V_k^\ast( \overline{x})\geq \alpha$.
    By \eqref{eq:boundedV}, we have $1_{\mathcal{T}_k}( \overline{x})\geq
    V_k^\ast( \overline{x})\geq \alpha>0 \Rightarrow \overline{x}\in
    \mathcal{T}_k$, which implies $\SRAsetk\subseteq \mathcal{T}_k$.
    The boundedness of $\SRAsetk$ follows by definition~\cite[Defn.    12.5.3]{TaoAnalysisII}.     \myqed

\subsection{Proof of Proposition~\ref{prop:Borel}}
\label{app:proof_Borel}

    \emph{a)} (By induction)
    By definition, the indicator functions
    $1_{\mathcal{T}_{k}}(\cdot)$ for all $k\in
    \mathbb{N}_{[0,N]}$ are Borel-measurable since
    $\mathcal{T}_{k}$ are Borel sets. Consequently,
    $V_N^\ast(\cdot)$ is Borel-measurable by \eqref{eq:VN}. 

    Consider the base case $k=N-1$. Since
    $V_{N}^\ast(\cdot)$ is Borel-measurable (by above) and
    bounded (by \eqref{eq:boundedV}) and the stochastic
    kernel is input-continuous, the function
    $\int_{\mathcal{X}}
    V_{N}^\ast(\overline{y})Q_{N-1}(d\overline{y}\vert\overline{x},\overline{u})$
    is continuous over $\mathcal{U}$ for each $
    \overline{x}\in\mathcal{X}$ by
    Definition~\ref{defn:cts_stoch}\ref{defn:cts_stoch_input}.  Since continuity
    implies upper
    semi-continuity~\cite[Lem.  7.13 (b)]{BertsekasSOC1978}
    and $ \mathcal{U}$ is compact, an optimal
    Borel-measurable input map $\mu_{N-1}^\ast(\cdot)$
    exists and $\int_{ \mathcal{X}}
    V_{N}^\ast(\overline{y})Q_{N-1}(d\overline{y}\vert\overline{x},\mu_{N-1}^\ast(\overline{x}))$
    is Borel-measurable  over $ \mathcal{X}$ by the
    measurable selection theorem~\cite[Thm. 2]{HimmelbergMOR1976}. 
    Finally, $V^\ast_{N-1}(\cdot)$ is Borel-measurable since the product operator
    preserves Borel-measurability~\cite[Cor. 18.5.7]{TaoAnalysisII}.

    For the case $k\in \mathbb{N}_{[0,N-2]}$,
    assume for induction that $V_{k+1}^\ast(\cdot)$ is
    Borel-measurable.  By the same arguments as above, a
    Borel-measurable $\mu_k^\ast(\cdot)$ exists and
    $V_{k}^\ast(\cdot)$ is Borel-measurable. 
    
    For every $k\in \mathbb{N}_{[0,N]}$, the set $\SRAsetk$
    is well-defined for $\alpha\in[0,1]$, since
    $V_k^\ast(\cdot)$ is well-defined. 

    \emph{b)} Since continuous stochastic kernels are
    input-continuous, the results of
    Proposition~\ref{prop:Borel}\ref{prop:Borel_exist}
    hold. 
    We know that $V_k^\ast(\cdot)$ is zero (and thereby,
    continuous) in the complement of
    $\mathcal{T}_k$ from \eqref{eq:VN} and
    \eqref{eq:Vt_recursionQ}. Also, \eqref{eq:VN} shows that
    $V_N(\cdot)$ is one (and thereby,
    continuous) in the relative interior of
    $\mathcal{T}_N$. We will show that
    $\int_{ \mathcal{X}}
    V_{k+1}^\ast(\overline{y})Q_k(d\overline{y}\vert\overline{x},\mu_k^\ast(
    \overline{x}))$ is continuous over $ \mathcal{X}$ for
    every $k\in \mathbb{N}_{[0,N-1]}$.
    Consequently, $V_k^\ast( \overline{x})$
    is continuous over the relative 
    interior of $\mathcal{T}_k$ for any
    $k\in \mathbb{N}_{[0,N-1]}$ by \eqref{eq:Vt_recursionQ},
    which completes the proof.
    
    The function 
    $\int_{ \mathcal{X}}
    V_{k+1}^\ast(\overline{y})Q_k(d\overline{y}\vert\overline{x},\overline{u})$
    is continuous over $ \mathcal{X}\times \mathcal{U}$ for
    every $k\in \mathbb{N}_{[0,N-1]}$ by
    Definition~\ref{defn:cts_stoch}\ref{defn:cts_stoch_all}, 
    continuous stochastic kernel $Q_k$, 
    and 
    Borel-measurability (Proposition~\ref{prop:Borel}\ref{prop:Borel_exist}) and
    boundedness \eqref{eq:boundedV} of
    $V_{k+1}^\ast(\cdot)$. Using induction and~\cite[Prop.
    7.32]{BertsekasSOC1978}, we conclude the existence of
    $\mu_k^\ast(\cdot)$ for each $k\in
    \mathbb{N}_{[0,N-1]}$, such that $\int_{
    \mathcal{X}}
    V_{k+1}^\ast(\overline{y})Q_k(d\overline{y}\vert\overline{x},\mu_k^\ast(
    \overline{x}))$ is lower and upper semi-continuous
    (thereby, continuous) over $ \mathcal{X}$.\myqed

\subsection{A lemma used in the proof of Proposition~\ref{prop:usc}}

\begin{lem}
    Let the system dynamics $f_k(\cdot),\ \forall k\in \mathbb{N}_{[0,N-1]}$ be
    continuous over $ \mathcal{X}\times \mathcal{U}\times \mathcal{W}$, sets
    $\mathcal{X}$ and $\mathcal{U}$ be closed. For every bounded, non-negative,
    and \usc{} function $h: \mathcal{X} \rightarrow
    \mathbb{R}$ and $k\in \mathbb{N}_{[0,N-1]}$, $\int_{
    \mathcal{X}} h(\overline{y}) Q_k(d\overline{y}\vert \overline{x},
    \overline{u})$ is \usc{} over $
    \mathcal{X}\times \mathcal{U}$.\label{lem:usc_intg}
\end{lem}
\begin{proof}
    By \eqref{eq:Q_defn_gen}, we write $\int_{ \mathcal{X}}
    h(\overline{y}) Q_k(d\overline{y}\vert \overline{x}, \overline{u})$ as the
    Lebesgue integral of $h(f_k(\overline{x}, \overline{u},\overline{w}))$ over
    $ \mathcal{W}$, with respect to the probability measure
    $\Prob_{\bw,k}$~\cite[Sec. 1.23]{RudinReal1987}, 
    \begin{align}
        \int_{ \mathcal{X}} h(\overline{y}) Q_k(d\overline{y}\vert \overline{x},
        \overline{u})=\int_{\mathcal{W}}h(f_k(\overline{x}, \overline{u},
        \overline{w}))d\Prob_{\bw,k}.
    \end{align}
    Note that $h(f_k(\overline{x}, \overline{u}, \overline{w}))$ is \usc{} over
    $ \mathcal{X}\times \mathcal{U}\times \mathcal{W}$, since the composition of
    a \usc{} function with a continuous function is \usc{} ~\cite[Ex. 1.4]{rockafellar_variational_2009}.
    Additionally, $h(f_k(\overline{x}, \overline{u}, \overline{w}))$ is bounded
    and non-negative, since $h(\cdot)$ is bounded and non-negative.
    If $L\in \mathbb{R}$ is an upper bound of $h(\cdot)$,
    then $ L-h(f_k(\overline{x}, \overline{u},
    \overline{w}))$ is non-negative and lower
    semi-continuous over $ \mathcal{X}\times \mathcal{U}$ for every $ \overline{ w}\in \mathcal{W}$.
    By Borel-measurability of $h$, $h(f_k(\overline{x}, \overline{u}, \bw))$ is
    Borel-measurable with respect to $(\mathcal{X}\times
    \mathcal{U}\times \mathcal{W})$.
    From Fatou's lemma~\cite[Sec. 6.2, Thm.
    2.1]{ChowProbability1997} and the fact that $
    L-h(f_k(\overline{x}, \overline{u}, \overline{w}))$ is
    lower semi-continuous, Borel-measurable, and
    non-negative, we have
    \begin{align}
        \liminf_{i}\int_{\mathcal{X}}(L-h(f_k(\overline{x}_i, \overline{u}_i,
        \overline{w})))d\Prob_{\bw,k} &\geq\int_{\mathcal{X}}\liminf_{i}(L-h(f_k(\overline{x}_i, \overline{u}_i, \overline{w})))d\Prob_{\bw,k} \nonumber \\%\psibwk( \overline{w})d\overline{w} \nonumber \\
        &\geq\int_{\mathcal{X}}(L-h(f_k(\overline{x}, \overline{u}, \overline{w})))d\Prob_{\bw,k} \label{eq:h_intg_lsc}  %
    \end{align}
    By linearity properties of the Lebesgue integral on
    \eqref{eq:h_intg_lsc}~\cite[Prop.
    19.2.6c]{TaoAnalysisII}, 
    $\limsup_{i}\int_{\mathcal{X}}h(f_k(\overline{x}_i,
    \overline{u}_i, \overline{w}))d\Prob_{\bw,k}$ is bounded
    from above by $\int_{\mathcal{X}}h(f_k(\overline{x}, \overline{u},
    \overline{w}))d\Prob_{\bw,k}$, proving that $\int_{
    \mathcal{X}} h(\overline{y}) Q_k(d\overline{y}\vert
    \overline{x}, \overline{u})$ is \usc, as desired.
\end{proof}

\subsection{Proof of Proposition~\ref{prop:usc}}
\label{app:proof_usc}

    \emph{a)} Since $ \mathcal{T}_k$ and $ \mathcal{X}$ are closed, $ 1_{\mathcal{T}_k}(\cdot),\ \forall k\in \mathbb{N}_{[0,N]}$ is \usc{} over $ \mathcal{X}$.
    Hence, $V_N^\ast(\cdot)$ is \usc{} over $ \mathcal{X}$.

    Consider the base case $k=N-1$. Since $ \mathcal{T}_N$
    is closed, $V_{N}^\ast(\cdot)$ is u.s.c., and $V_{N}^\ast(\cdot)$ is bounded and non-negative by \eqref{eq:boundedV}.
    Hence, $\int_{ \mathcal{X}}
    V_{N}^\ast(\overline{y})Q_{N-1}(d\overline{y}\vert\overline{x},\overline{u})$
    is \usc{} over $ \mathcal{X}\times \mathcal{U}$ by Lemma~\ref{lem:usc_intg}.
    By a selection result for semicontinuous cost
    functions~\cite[Prop. 7.33]{BertsekasSOC1978} and
    compactness of $ \mathcal{U}$, an optimal
    Borel-measurable input map $\mu_{N-1}^\ast(\cdot)$
    exists and $\int_{ \mathcal{X}}
    V_{N}^\ast(\overline{y})Q_{N-1}(d\overline{y}\vert\overline{x},\mu_{N-1}^\ast(\overline{x}))$ is \usc{} over $ \mathcal{X}$. 
    Since upper semicontinuity is preserved under multiplication~\cite[Props. B.1]{PuttermanMarkov2005}, $V_{N-1}^\ast(\cdot)$ is \usc{} over $ \mathcal{X}$ by \eqref{eq:Vt_recursionQ}.

    For the case $k\in \mathbb{N}_{[0,N-2]}$, assume for induction that $V_{k+1}^\ast(\cdot)$ is \usc{}.  
    By the same arguments as above, a Borel-measurable $\mu_k^\ast(\cdot)$
    exists and $V_{k}^\ast(\cdot)$ is \usc{}.
 
    For every $k\in \mathbb{N}_{[0,N]}$, the upper
    semi-continuity of $V_k^\ast(\cdot)$ implies that the
    set $\SRAsetk$ is closed for $\alpha\in[0,1]$.

    \emph{b)} By Propositions~\ref{prop:contains}
    and~\ref{prop:usc}\ref{prop:usc_closed} and Heine-Borel
    theorem.  \myqed

\subsection{Proof of Proposition~\ref{prop:W_prop}}
\label{app:proof_W_prop}
    \emph{\ref{prop:W_prop_usc})} Similarly to the proof of
    Proposition~\ref{prop:usc}
    (Appendix~\ref{app:proof_usc}).

    \emph{\ref{prop:W_prop_OL_use})}
    By induction, we have $W_0( \overline{x}, \overline{U})\leq
    V_0^\ast( \overline{x})$ for every $ \overline{x}\in \mathcal{X}$ and $
    \overline{U}\in \mathcal{U}^N$.  
    Consequently, $W_0^\ast( \overline{x})=\sup_{ \overline{U}\in
    \mathcal{U}^N}W_0( \overline{x}, \overline{U})\leq V_0^\ast( \overline{x})$
    by \eqref{eq:W_ast} and the definition of the supremum.
    Consequently, $\FSRAset\subseteq \SRAset$ for every
    $\alpha\in[0,1]$ by definition.%

    \emph{\ref{prop:W_prop_convex})} Similarly to the proof of Theorem~\ref{thm:convex}, we can show by induction that $W_0( \overline{x}_0, \overline{U})$ is log-concave~in $ \mathcal{X}\times\mathcal{U}^N$ when $ \mathcal{T}_k$ is convex and $\psibwk$ is log-concave.
    Additionally, $ \mathcal{U}^N$ is convex since $ \mathcal{U}$ is convex~\cite[Sec. 2.3.2]{BoydConvex2004}.
    Hence, \eqref{eq:W_ast} (and thereby \eqref{prob:OL}) is a log-concave optimization.
    Since partial supremum over convex sets preserves log-concavity~\cite[Sec.
    3.2, 3.5]{BoydConvex2004}, $W_0^\ast(\cdot)$ is log-concave over $
    \mathcal{X}$. %

    \emph{\ref{prop:W_prop_cvx_cmpt})} By Proposition~\ref{prop:W_prop}\ref{prop:W_prop_usc}
    and~\ref{prop:W_prop}\ref{prop:W_prop_convex}, $\FSRAset$ for every
    $\alpha\in(0,1]$ is convex and closed by \eqref{eq:FstochRAset}. By
    Proposition~\ref{prop:cvx_cmpt}, $\SRAset$ for every $\alpha\in(0,1]$ is
    convex and compact. The proof follows from the fact that closed
    subsets of compact sets are compact~\cite[Cor. 12.5.6]{TaoAnalysisII}. \myqed


\begin{thebibliography}{10}

\bibitem{AbateHSCC2007}
A.~Abate, S.~Amin, M.~Prandini, J.~Lygeros, and S.~Sastry.
\newblock Computational approaches to reachability analysis of stochastic
  hybrid systems.
\newblock In {\em Proc. Hybrid Syst.: Comput. \& Ctrl.}, pages 4--17, 2007.

\bibitem{Abate2020ARCH}
A.~Abate et~al.
\newblock {ARCH-COMP20} category report: Stochastic modelling.
\newblock In {\em Proc. Int'l W. on App. Verification Cont. \& Hybrid Syst.},
  pages 76--106, 2020.

\bibitem{AbateAutomatica2008}
A.~Abate, M.~Prandini, J.~Lygeros, and S.~Sastry.
\newblock Probabilistic reachability and safety for controlled discrete time
  stochastic hybrid systems.
\newblock {\em Automatica}, 44(11):2724--2734, 2008.

\bibitem{mosek}
MOSEK ApS.
\newblock {\em MOSEK optimization toolbox (9.1.9)}, 2020.

\bibitem{bertsekas_infinite_1972}
D.~Bertsekas.
\newblock Infinite time reachability of state-space regions by using feedback
  control.
\newblock {\em {IEEE} Trans. Autom. Ctrl.}, 17(5):604--613, 1972.

\bibitem{bertsekas1971minimax}
D.~Bertsekas and I.~Rhodes.
\newblock On the minimax reachability of target sets and target tubes.
\newblock {\em Automatica}, 7(2):233--247, 1971.

\bibitem{BertsekasSOC1978}
D.~Bertsekas and S.~Shreve.
\newblock {\em Stochastic optimal control: The discrete time case}.
\newblock Academic Press, 1978.

\bibitem{BoydConvex2004}
S.~Boyd and L.~Vandenberghe.
\newblock {\em Convex optimization}.
\newblock Cambridge Univ. Press, 2004.

\bibitem{cauchi2019stochy}
N.~Cauchi and A.~Abate.
\newblock \texttt{StocHy}: Automated verification and synthesis of stochastic
  processes.
\newblock In {\em Proc. Int'l Conf. Tools \& Alg. Constr. \& Analysis of
  Syst.}, pages 247--264, 2019.
\newblock \url{https://github.com/natchi92/stochy}.

\bibitem{cauchi2019efficiency}
N.~Cauchi, L.~Laurenti, M.~Lahijanian, A.~Abate, M.~Kwiatkowska, and
  L.~Cardelli.
\newblock Efficiency through uncertainty: Scalable formal synthesis for
  stochastic hybrid systems.
\newblock In {\em Proc. Hybrid Syst.: Comput. \& Ctrl.}, pages 240--251, 2019.

\bibitem{chatterjee2011maximizing}
D.~Chatterjee, E.~Cinquemani, and J.~Lygeros.
\newblock Maximizing the probability of attaining a target prior to extinction.
\newblock {\em Nonlinear Analysis: Hybrid Syst.}, 5(2):367--381, 2011.

\bibitem{chatterjee2011stochastic}
D.~Chatterjee, P.~Hokayem, and J.~Lygeros.
\newblock Stochastic receding horizon control with bounded control inputs: A
  vector space approach.
\newblock {\em {IEEE} Trans. Autom. Ctrl.}, 56(11):2704--2710, 2011.

\bibitem{ChowProbability1997}
Y.S. Chow and H.~Teicher.
\newblock {\em Probability Theory: Independence, Interchangeability,
  Martingales}.
\newblock Springer New York, 1997.

\bibitem{dharmadhikari1988unimodality}
S.~Dharmadhikari and K.~Joag-Dev.
\newblock {\em Unimodality, convexity, and applications}.
\newblock Elsevier, 1988.

\bibitem{DingAutomatica2013}
J.~Ding, M.~Kamgarpour, S.~Summers, A.~Abate, J.~Lygeros, and C.~Tomlin.
\newblock A stochastic games framework for verification and control of discrete
  time stochastic hybrid systems.
\newblock {\em Automatica}, 49(9):2665--2674, 2013.

\bibitem{SDP_kariotgolou}
D.~Drzajic, N.~Kariotoglou, M.~Kamgarpour, and J.~Lygeros.
\newblock A semidefinite programming approach to control synthesis for
  stochastic reach-avoid problems.
\newblock In {\em Proc. Int'l W. on App. Verification Cont. \& Hybrid Syst.},
  pages 134--143, 2016.

\bibitem{farina2016stochastic}
M.~Farina, L.~Giulioni, and R.~Scattolini.
\newblock Stochastic linear model predictive control with chance constraints--a
  review.
\newblock {\em J. Process Ctrl.}, 44:53--67, 2016.

\bibitem{GenzJCGS1992}
A.~Genz.
\newblock Numerical computation of multivariate normal probabilities.
\newblock {\em J. Comp. \& Graph. Stat.}, 1(2):141--149, 1992.

\bibitem{GleasonCDC2017}
J.~Gleason, A.~Vinod, and M.~Oishi.
\newblock Underapproximation of reach-avoid sets for discrete-time stochastic
  systems via {L}agrangian methods.
\newblock In {\em Proc. IEEE Conf. Dec. \& Ctrl.}, pages 4283--4290, 2017.

\bibitem{CVX}
M.~Grant and S.~Boyd.
\newblock \texttt{CVX}: \texttt{MATLAB} software for disciplined convex
  programming.
\newblock \url{http://cvxr.com/cvx}.

\bibitem{GubnerProbability2006}
J.~Gubner.
\newblock {\em Probability and random processes for electrical and computer
  engineers}.
\newblock Cambridge Univ. Press, 2006.

\bibitem{MPT3}
M.~Herceg, M.~Kvasnica, C.N. Jones, and M.~Morari.
\newblock {Multi-Parametric Toolbox 3.0}.
\newblock In {\em Proc. Euro. Ctrl. Conf.}, pages 502--510, 2013.
\newblock \url{https://www.mpt3.org/}.

\bibitem{HimmelbergMOR1976}
C.~Himmelberg, T.~Parthasarathy, and F.~VanVleck.
\newblock Optimal plans for dynamic programming problems.
\newblock {\em Mathematics of Operations Research}, 1(4):390--394, 1976.

\bibitem{kariotoglou2017linear}
N.~Kariotoglou, M.~Kamgarpour, T.~Summers, and J.~Lygeros.
\newblock The linear programming approach to reach-avoid problems for markov
  decision processes.
\newblock {\em J. Artificial Intelligence Research}, 60:263--285, 2017.

\bibitem{KariotoglouSCL2016}
N.~Kariotoglou, K.~Margellos, and J.~Lygeros.
\newblock On the computational complexity and generalization properties of
  multi-stage and stage-wise coupled scenario programs.
\newblock {\em Syst. \& Ctrl. Lett.}, 94:63--69, 2016.

\bibitem{kolda2003optimization}
T.~Kolda, R.~Lewis, and V.~Torczon.
\newblock Optimization by direct search: New perspectives on some classical and
  modern methods.
\newblock {\em SIAM review}, 45(3):385--482, 2003.

\bibitem{LesserCDC2013}
K.~Lesser, M.~Oishi, and R.~Erwin.
\newblock Stochastic reachability for control of spacecraft relative motion.
\newblock In {\em Proc. IEEE Conf. Dec. \& Ctrl.}, pages 4705--4712, 2013.

\bibitem{MaloneHSCC2014}
N.~Malone, K.~Lesser, M.~Oishi, and L.~Tapia.
\newblock Stochastic reachability based motion planning for multiple moving
  obstacle avoidance.
\newblock In {\em Proc. Hybrid Syst.: Comput. \& Ctrl.}, pages 51--60, 2014.

\bibitem{ManganiniCYB2015}
G.~Manganini, M.~Pirotta, M.~Restelli, L.~Piroddi, and M.~Prandini.
\newblock Policy search for the optimal control of {Markov} {Decision}
  {Processes}: A novel particle-based iterative scheme.
\newblock {\em {IEEE} Trans. Cybern.}, 46(11):2643 -- 2655, 2015.

\bibitem{mesbah2016stochastic}
A.~Mesbah.
\newblock Stochastic model predictive control: An overview and perspectives for
  future research.
\newblock {\em IEEE Ctrl. Syst. Mag.}, 36(6):30--44, 2016.

\bibitem{OnoCDC2008}
M.~Ono and B.~Williams.
\newblock Iterative risk allocation: a new approach to robust {Model}
  {Predictive} {Control} with a joint chance constraint.
\newblock In {\em Proc. IEEE Conf. Dec. \& Ctrl.}, pages 3427--3432, 2008.

\bibitem{PuttermanMarkov2005}
M.~Putterman.
\newblock {\em Markov {Decision} {Processes}: {Discrete} {Stochastic} {Dynamic}
  {Programming}}.
\newblock John Wiley \& Sons, 2005.

\bibitem{rockafellar_variational_2009}
R.~Rockafellar and R.~Wets.
\newblock {\em Variational analysis}, volume 317.
\newblock Springer Science \& Business Media, 2009.

\bibitem{RudinReal1987}
W.~Rudin.
\newblock {\em Real and complex analysis}.
\newblock Tata McGraw-Hill, 1987.

\bibitem{sartipizadeh2019}
H.~Sartipizadeh, A.~Vinod, Beh{\c{c}}et A{\c{c}}ikmese, and M.~Oishi.
\newblock Voronoi partition-based scenario reduction for fast sampling-based
  stochastic reachability computation of {LTI} systems.
\newblock In {\em Proc. Amer. Ctrl. Conf.}, pages 37--44, 2019.

\bibitem{SkafTAC2010}
J.~Skaf and S.~Boyd.
\newblock Design of affine controllers via convex optimization.
\newblock {\em {IEEE} Trans. Autom. Ctrl.}, 55(11):2476--87, 2010.

\bibitem{soudjani2014probabilistic}
S.~Soudjani and A.~Abate.
\newblock Probabilistic reach-avoid computation for partially degenerate
  stochastic processes.
\newblock {\em {IEEE} Trans. Autom. Ctrl.}, 59(2):528--534, 2014.

\bibitem{soudjani2015faust}
S.~Soudjani, C.~Gevaerts, and A.~Abate.
\newblock \texttt{FAUST}$^2$: {F}ormal {A}bstractions of {U}ncountable-{ST}ate
  {ST}ochastic processes.
\newblock In {\em Proc. Int'l Conf. Tools \& Alg. Constr. \& Analysis of
  Syst.}, pages 272--286, 2015.

\bibitem{SummersAutomatica2010}
S.~Summers and J.~Lygeros.
\newblock Verification of discrete time stochastic hybrid systems: {A}
  stochastic reach-avoid decision problem.
\newblock {\em Automatica}, 46(12):1951--1961, 2010.

\bibitem{TaoAnalysisII}
T.~Tao.
\newblock {\em Analysis II}.
\newblock Hindustan Book Agency, 2 edition, 2009.

\bibitem{sreachtools}
A.~Vinod, J.~Gleason, and M.~Oishi.
\newblock \texttt{SReachTools}: A \texttt{MATLAB} {S}tochastic {R}eachability
  {T}oolbox.
\newblock In {\em Proc. Hybrid Syst.: Comput. \& Ctrl.}, pages 33 -- 38, 2019.
\newblock \url{https://sreachtools.github.io}.

\bibitem{VinodLCSS2017}
A.~Vinod and M.~Oishi.
\newblock Scalable underapproximation for the stochastic reach-avoid problem
  for high-dimensional {LTI} systems using {Fourier} transforms.
\newblock {\em {IEEE} Lett.-Contr. Syst. Soc.}, 1(2):316--321, 2017.

\bibitem{VinodHSCC2018}
A.~Vinod and M.~Oishi.
\newblock Scalable underapproximative verification of stochastic {LTI} systems
  using convexity and compactness.
\newblock In {\em Proc. Hybrid Syst.: Comput. \& Ctrl.}, pages 1--10, 2018.

\bibitem{VinodACC2019}
A.~Vinod, V.~Sivaramakrishnan, and M.~Oishi.
\newblock Piecewise-affine approximation-based stochastic optimal control with
  {G}aussian joint chance constraints.
\newblock In {\em Proc. Amer. Ctrl. Conf.}, pages 2942--2949, 2019.

\bibitem{webster1994convexity}
R.~Webster.
\newblock {\em Convexity}.
\newblock Oxford University Press, 1994.

\bibitem{wiesel1989_spaceflight}
W.~Wiesel.
\newblock {\em Spaceflight Dynamics}.
\newblock McGraw-Hill, NY, 1989.

\bibitem{YangAutomatica2017}
I.~Yang.
\newblock A dynamic game approach to distributionally robust safety
  specifications for stochastic systems.
\newblock {\em Automatica}, 94:94--101, 2018.

\end{thebibliography}
\end{document}